\documentclass[11pt]{article}
\usepackage[normalem]{ulem}

\usepackage{amsmath,amssymb,amsthm,color}
\usepackage{algorithm,algorithmicx,algpseudocode}
\usepackage{subcaption}
\usepackage[hidelinks,colorlinks=true,linkcolor=blue,citecolor=black]{hyperref}
\usepackage{fancybox,graphicx,bm,float}            

\usepackage[round]{natbib}

\newtheorem{theorem}{Theorem}
\newtheorem{lemma}[theorem]{Lemma}
\newtheorem{corollary}[theorem]{Corollary}
\newtheorem{proposition}[theorem]{Proposition}

\newtheorem{remark}{Remark}[section]

\newcommand{\R}{\mathbb{R}}

\newcommand{\bpmat}{\begin{pmatrix}}
\newcommand{\epmat}{\end{pmatrix}}
\newcommand{\Symn}{\mathcal{S}^{n \times n}}

\newcommand{\dist}{\text{dist}}
\DeclareMathOperator*{\argmin}{argmin}

\newcommand{\Rmnum}[1]{\uppercase\expandafter{\romannumeral #1}} 
\usepackage[raggedright]{titlesec}
\titleformat{\chapter}{\centering\Huge\bfseries}{Chapter \Rmnum{\thechapter} }{1em}{} 

\usepackage{mathrsfs} 
\usepackage{enumerate} 

\usepackage{multirow}
\usepackage{stmaryrd}

\title{ 
Improved Convergence Factor of Windowed Anderson Acceleration for Symmetric Fixed-Point Iterations
}

\author{ Casey Garner\thanks{School of Mathematics, University of Minnesota (\href{mailto:garne214@umn.edu}{garne214@umn.edu}, \href{mailto:lerman@umn.edu}{lerman@umn.edu}) }
\hspace{1cm}
Gilad Lerman\footnotemark[1]
\hspace{1cm}
Teng Zhang\thanks{Department of Mathematics, University of Central Florida (\href{mailto:Teng.Zhang@ucf.edu}{Teng.Zhang@ucf.edu})}}

\date{\today}

\begin{document}
\maketitle

\vspace{-0.2in}

\begin{abstract}
{
This paper studies the commonly utilized windowed Anderson acceleration (AA) algorithm for fixed-point methods, $x^{(k+1)}=q(x^{(k)})$. It provides the first proof that when the operator $q$ is linear and symmetric the windowed AA, which uses a sliding window of prior iterates, improves the root-linear convergence factor over the fixed-point iterations. When $q$ is nonlinear, yet has a symmetric Jacobian at a fixed point, a slightly modified AA algorithm is proved to have an analogous root-linear convergence factor improvement over fixed-point iterations. 
Simulations verify our observations. Furthermore, experiments with different data models demonstrate AA is significantly superior to the standard fixed-point methods for Tyler's M-estimation.

\vspace{3mm}
    \noindent\textbf{Keywords:} 
Anderson acceleration, fixed-point method, Tyler's M-estimation

\vspace{3mm}
    \noindent\textbf{MSC codes:} 
     65F10, 65H10, 68W40
}

\end{abstract}

\section{Introduction}\label{sec:intro}
This paper investigates the convergence properties of the popular acceleration technique for fixed-point problems called Anderson acceleration (AA). The method, originally introduced by \cite{anderson1965} in the context of integral equations, utilizes a history of prior iterates to accelerate the classical fixed-point iteration, aka the Picard iteration \citep{picard1890memoire}. 
A brief history of the method, including its relationship to Pulay mixing, nonlinear GMRES, and quasi-Newton methods, can be found in the literature \citep{fang2009two,kelley2018numerical,potra2013characterization,pulay1980convergence}.
A renewed interest in AA came about following the work of \cite{walker2011anderson} where they demonstrated the effectiveness of AA in numerous applications such as nonnegative matrix factorization and domain decomposition. 
Recently, AA, in various forms, has been studied in a myriad of settings and applied in numerous applications \citep{geist2018anderson,shi2019regularized,sun2021damped,zhang2020globally,brezinski2022shanks,bian2021anderson,bian2022anderson}. 

Despite a long history of use and strong recent interest, the accelerated convergence of AA is still not completely understood. 
The first mathematical convergence results for AA, for linear and nonlinear problems, were established by \cite{Toth2015}; however, they only proved AA did not worsen the convergence of the fixed-point iteration. Their theory did
not prove AA converged faster than the fixed-point method, as often witnessed in practice. Following Toth and Kelley, further efforts have shed more light on the convergence behavior of AA. In the papers by \cite{Evans2020}, \cite{pollock2019anderson}, and \cite{pollock2021anderson}, they proved AA improved convergence by studying the ``stage-k gain" of AA, which shows how much AA improves upon the fixed-point method at iteration $k$. Their analysis showed AA can surpass the fixed-point iteration by a factor $0\leq \theta_k \leq 1$ known as the gain at step $k$; however, 
as noted by De Sterck and He, ``it is not clear how $\theta_k$ may be evaluated or bounded in practice or how it may translate to an improved linear asymptotic convergence factor (\citeyear{doi:10.1137/21M1449579})." 
In this work, De Sterck and He note some intriguing properties of AA, such as how the linear asymptotic convergence factor depends on the initialization of the method. 
De Sterck, He, and Krzysik continuing this study related windowed AA to Krylov methods and developed new results on residual convergence bounds for linear fixed-point iterations (\citeyear{de2024anderson}); however, they note in their conclusion that bounding the asymptotic convergence factors remains an open and difficult question. In a different direction, \cite{rebholz2023effect} proved how AA affects the convergence rate of superlinearly and sublinearly converging fixed-point iterations.  

The aforementioned papers have added much to our understanding of the convergence of AA in the past decade, and this paper joins the conversation by establishing the improved root-linear convergence factor of AA over the fixed-point iteration {\color{black} in the case of iterations with symmetry}. To the best of our knowledge, this paper is the first to answer this open question and present a bound which clearly shows AA has a better asymptotic convergence factor than the fixed-point iteration.
Notably, our approach to proving  the claim is novel and does not utilize the notation of gain as defined in prior works, see \cite{Evans2020,pollock2019anderson,pollock2021anderson}.
Furthermore, our results only require common assumptions found in the literature. For example, we assume the Jacobian matrix of the operator ${q}$, for which we desire to compute fixed points, is symmetric at fixed points \citep{scieur2019generalized,NIPS2016_bbf94b34}.   
{\color{black} It is worth noting that symmetric Jacobian matrices often arise in fixed-point problems related to optimization. For example, minimizing $f(x)$  corresponds to the fixed-point problem $\nabla f(x)=0$. The Jacobian of $\nabla f(x)$ is the Hessian of $f$, which is naturally symmetric.}

It is also worth noting that positive asymptotic convergence results have been proven for other variants of AA. Recently, \cite{DeSterck2024} presented a characterization of the convergence for a restarted version of AA. 
Their results fully describe the behavior of a restarted AA iteration for solving $2\times 2$ linear systems. 
This paper contains similar results to an earlier work, \citep{Both2019}, where the authors proved a restarted version of AA has a better asymptotic convergence factor compared to the fixed-point method. 
Our work differs substantially because we do not study a restarted AA method but a windowed AA method, which uses a sliding window of prior iterates to compute the next iterate. We compare different AA variants in Section~\ref{sec:compare}.

\noindent
\textbf{Definitions and notation:} $\R^{n\times n}$ and $\mathcal{S}^{n\times n}$ denote the set of real $n$-by-$n$ and real symmetric $n$-by-$n$ matrices respectively. Given $A\in\R^{n\times n}$, its $\ell_2$ norm (or spectral/operator norm) is denoted $\|A\|$. 
The largest and the smallest eigenvalues of a matrix $A$ are denoted $\lambda_{\max}(A)$ and $\lambda_{\min}(A)$ respectively. 
Given the vector $x \in \R^n$, $\text{Diag}(x)$ denotes the diagonal matrix formed by $x$. 
For any three points $O, P, Q\in\R^n$, $\angle(OPQ)\in [0,\pi]$ represents the angle between the vectors $\vec{PO}$ and $\vec{PQ}$. For any sequence $\{x^{(k)}\}$ that converges to $x_*$, we define its  $r$-linear convergence factor to be  
\[
\rho_{\{x^{(k)}\}}=\lim\sup_{k\rightarrow\infty}\|x_*-x^{(k)}\|^{\frac{1}{k}}.
\]
Then, we say $\{x^{(k)}\}$ converges r-linearly with r-linear convergence factor $\rho_{\{x^{(k)}\}}$, and here the prefix ``r-'' stands for ``root''. 

\noindent
\textbf{Organization of the rest of the paper:} 
Section~\ref{sec:AA} introduces the problem setting and defines the AA algorithm we study; 
Section~\ref{sec:main} states the main results of the paper with the key result being a proven upper bound on the root-linear convergence factor for AA when the Jacobian matrix is symmetric; Section~\ref{sec:simulations} presents numerical experiments for both linear and nonlinear operators and compares the performance of different types of AA variants; Section \ref{sec:proofs} contains the proofs of our results; the paper concludes in Section~\ref{sec:conclusion} with directions for future inquiry.

\section{Anderson Acceleration}\label{sec:AA}
This paper concerns the convergence of acceleration methods for computing fixed points of an operator ${q}:\R^n \rightarrow \R^n$, i.e, computing points $x_* \in \R^n$ such that 
$
x_*=q(x_*).
$
For this problem, the standard fixed-point method is the iterative procedure described by 
\begin{equation}\label{eq:fp1}
x^{(k+1)}=q(x^{(k)}).
\end{equation}
In this work, we study the windowed AA algorithm with depth $m$, AA(m), applied to the fixed-point iteration \eqref{eq:fp1}. 
A full description of AA(m) is given in Algorithm~\ref{alg:aa}. 
The key idea behind the method is to use a history of at most $m+1$ points to construct the next update at each iteration. 
A linear least-squares problem, \eqref{eq:aa_optimization}, is solved at each iteration to determine how to utilize the prior iterates to form the next update. 
As discussed in Section \ref{sec:intro}, AA(m) demonstrated significant effectiveness in many applications and this practical utility has motivated the resurgence of interest in the convergence analysis of this method. 
\begin{algorithm}[tb]
\caption{{{Windowed Anderson acceleration algorithm of depth $m$; AA(m)}
}}  
\label{alg:aa}
\begin{flushleft} 
  {\bf Input:}  Operator $q: \R^n\rightarrow\R^n$, initialization $x^{(0)}\in \R^n$, depth $m\in\{1,2,\hdots\}$\\
  {\bf Output:} An {\color{black}approximation of a} fixed point of $q$.\\
  {\bf Steps:}\\
  {\bf 1:} {\color{black}Set $x^{(1)}=q(x^{(0)})$ and $k=1$.}\\ 
  {\bf 2:} Let $m_k=\min(m,k)$. Solve the minimization problem 
\begin{equation}\label{eq:aa_optimization}
\min_{\sum_{j=k-m_k}^{k} \alpha_j^{(k)}=1}\Big\|\sum_{j=k-m_k}^{k}\alpha_j^{(k)}\big(q(x^{(j)})-x^{(j)}\big)\Big\|,
\end{equation}
and let 
\begin{equation}\label{eq:aa_update}
x^{(k+1)}=\sum_{j=k-m_k}^{k}\alpha_j^{(k)} q(x^{(j)}).\nonumber
\end{equation}
  \\
  {\bf 3:} {\color{black}Set $k=k+1$ and return to Step 2.}
  \end{flushleft} 
  \end{algorithm}

\section{Main Results}\label{sec:main}
This section presents our main results showcasing AA(m) has an improved r-linear convergence factor over the fixed-point iteration. Section~\ref{sec:linear} shows that in the case of linear and symmetric operators, AA(m) converges faster than the fixed-point iteration with {\color{black} an upper bound on the convergence factor} given in Theorem~\ref{thm:main1}. {\color{black}It also investigates the upper bound for special cases in Propositions~\ref{cor:special} and~\ref{cor:main}, although a detailed geometric understanding of the bound is deferred to Lemma~\ref{lemma:w0} in Section~\ref{sec:proofs}.} Section \ref{sec:reformulate_AA} describes a reformulation of AA(m), an idea that makes the proofs more elegant. 
Section~\ref{sec:nonlinear} shows when $q$ is nonlinear, yet has
a symmetric Jacobian at the solution, a
modified version of Algorithm~\ref{alg:aa} has an analogous root-linear convergence factor improvement over the fixed-point iteration when initialized near a solution.
Section~\ref{sec:discussion}  concludes with a discussion of our results. 

\subsection{Linear Symmetric Operators}\label{sec:linear}
We begin with the convergence of AA(m) for linear symmetric operators $q$, where $q(x)=Wx+a$ with
$W \in \mathcal{S}^{n\times n}$ and $-1< \lambda_{\min}(W)\leq \lambda_{\max}(W)< 1$. 
{\color{black}Theorem~\ref{thm:main1} establishes 
an upper bound on the $r$-linear convergence factor of AA(m), dependent upon  $\lambda_{\min}(W)$ and $\lambda_{\max}(W)$, and demonstrates a strict improvement over the convergence factor of the fixed-point iteration if $\lambda_{\max}(W) \neq -\lambda_{\min}(W)$.}

\begin{theorem}[Convergence factor of AA(m) for linear symmetric operators]\label{thm:main1}
Let $q(x) = {W}x + {a}$ with $W \in \mathcal{S}^{n\times n}$ and $\|{W}\|<1$. Then 
AA(m) converges $r$-linearly {\color{black}to the unique fixed point} of $q$ with {\color{black}a convergence factor bounded above} by $\sqrt{w_0 \|W\|}$ for any initialization,
where
{\color{black}
\begin{multline}\label{eq:w0}
w_0 = \sup_{\delta\geq 0}\;\;\sin\!\!\Bigg[\!
\sin^{-1}\!\!\Big(\!\frac{|\lambda_{\max}(W)-\lambda_{\min}(W)|\delta}{\sqrt{\!4+\!\big(2\!-\!\lambda_{\max}(\!W\!)\!-\!\lambda_{\min}(W\!)\big)^2\!\delta^2}}\!\Big)\! \\
+ \Big|\tan^{-1}(\delta)-\tan^{-1}\Big(\big(1-\frac{\lambda_{\max}(W)+\lambda_{\min}(W)}{2}\big)\delta\Big)\Big|\Bigg].
\end{multline}
}
 More specifically, if $\{x^{(k)}\}$ is the sequence generated by AA(m), then for all $k\geq 2$
 \begin{equation}\label{eq:twoiterations}
\frac{\|(W-I)(x^{(k+1)}-x_*)\|}{\|(W-I)(x^{(k-1)}-x_*)\|}\leq w_0 \|W\|.\end{equation}
Here $w_0\leq \|W\|$ 
with equality holds if and only if $\lambda_{\max}(W)= - \lambda_{\min}(W)$.
\end{theorem}
{\color{black}Theorem~\ref{thm:main1} implies that the $r$-linear convergence factor of AA(m) is bounded above by  $\sqrt{w_0\|W\|}$, which is less than or equal to $\|W\|$, the r-linear convergence factor of the fixed-point iteration. Furthermore, the theorem indicates that when  
$\lambda_{\max}(W)\neq -\lambda_{\min}(W)$, the convergence factor of AA(m) is strictly smaller.

Although the expression for $w_0$ in \eqref{eq:w0} is complicated, we note an equivalent geometric interpretation of $w_0$ will be presented in Lemma~\ref{lemma:w0} and visualized in Figure~\ref{fig:visualization_proof2}. Additionally, the following proposition shows that in the special case where $W$ is a scalar matrix, i.e., $W= \|W\|I$, $w_0$ has a simple, explicit description which depends on $\|W\|$.
}

{\color{black}
\begin{proposition}[Explicit expression of $w_0$ for a special case] \label{cor:special}
If $W$ is a scalar matrix with $\|W\|<1$, then $w_0$ defined in \eqref{eq:w0} can be concisely expressed as $w_0=\|W\|/\sqrt{2-\|W\|}$. 
\end{proposition}
}
We extend {\color{black}Theorem~\ref{thm:main1}} by showing if $m\geq 2$, 
the convergence factor of AA(m) is always strictly better than the {\color{black} convergence factor} of the fixed-point method.
\begin{proposition}
[Improved $r$-linear convergence factor]
\label{cor:main}
Assume $W \in \Symn$ {\color{black}with $\|W\|<1$}. If either $m\geq 2$ or $\lambda_{\max}(W)\neq -\lambda_{\min}(W)$, then 
 the
r-linear convergence factor of AA(m) is strictly smaller than $\|W\|$, which is the r-linear convergence factor of the fixed-point iteration. 
If on the other hand  
$m =1 $ and $\lambda_{\max}(W)=-\lambda_{\min}(W)$, then the convergence factor of AA(m) may be $\|W\|$ for {\color{black}certain  matrices $W$ and specific initial estimates  $x^{(0)}$}. 
\end{proposition}

\subsection{A Helpful Reformulation of Anderson Acceleration}
\label{sec:reformulate_AA}

{\color{black}
For our analysis it is beneficial to examine an equivalent reformulation of AA(m). 
In order to present our reformulation, we utilize the following result about two particular sequences generated by AA(m).
\begin{theorem}\label{thm:AA_invariance}
Let $f:\R^n \rightarrow \R^n$, $m \in \{1,2,\hdots\}$, and $b, x^{(0)} \in \R^n$.
Define the operators $q_1(x) := f(x) + x$ and $q_{2}(x):= f(x-b)+x$ and 
apply AA(m) on $q_1$ initialized at $x^{(0)}$ to generate the sequence $\{x^{(k)}_1\}$ and $q_{2}$ initialized at $x^{(0)} +b$ to generate the sequence $\{x^{(k)}_{2}\}$. 
Then the sequences 
$\{x^{(k)}_1\}$ and $\{x^{(k)}_{2}\}$ are identical up to a shift by $b$, that is, $x^{(k)}_{2} = x^{(k)}_{1} + b$ for all $k \geq 1$. 
\end{theorem}

Theorem~\ref{thm:AA_invariance} is proven in Section~\ref{sec:proofs}. 
A key consequence of it is that the sequences $\{x_1^{(k)}\}$ and $\{x_2^{(k)}\}$ exhibit identical convergence behavior. Specifically, if one sequence converges, so does the other, with identical convergence rates and convergence factors. Moreover, when $q(x)$ is a linear function, Theorem~\ref{thm:AA_invariance} 
enables a useful simplification that we leverage in our proof. 
\begin{corollary}\label{cor:simplify_linear}
Let $W\in \Symn$ with $\|W\|<1$ and $a \in \R^n$. If AA(m) is applied to $q_1(x) = Wx$ and $q_2(x) = Wx+a$ as described in Theorem~\ref{thm:AA_invariance}, then the generated sequences are the same up to a shift. In addition,  $q_1$ has a unique fixed point at $x = 0$. 
\end{corollary}
\begin{proof}
Define $Q:=W-I$, $b:=-Q^{-1}a$ ($Q$ must be invertible because $\|W\| <1$), and $f(x) := Qx$,  
then $q_1(x) = f(x)+x$ and $q_2(x) =f(x-b) + x$. Corollary~\ref{cor:simplify_linear} then follows by 
Theorem~\ref{thm:AA_invariance}. 
In addition, $x = 0$ is the unique fixed-point of $q_1$ because $Q$ is invertible.
\end{proof}
Thus, by Corollary~\ref{cor:simplify_linear}, to analyze the convergence of AA(m) in the linear setting, we may assume without loss of generality that ${a}={0}$, $q(x)=Wx$, and ${x}_* = {0}$.

Second, we present our reformulation of AA(m) based on this assumption.} Let $\widetilde{x}^{(k)}=q(x^{(k)})-x^{(k)}$, then AA(m) in Algorithm~\ref{alg:aa} is equivalent to the iterative update of $\widetilde{x}$ as follows:
\begin{enumerate}
\item Let $\widetilde{y}^{(k+1)}$ be the point on the affine subspace spanned by $\widetilde{x}^{(k-m_k)}, \cdots, \; \widetilde{x}^{(k)}$ with the smallest norm, that is,
\begin{equation}\label{eq:tildey_reformulation}
\color{black}
\widetilde{y}^{(k+1)}=\argmin \|y\|,\,\,\text{subject to }y=\sum_{j=k-m_k}^{k}\alpha_j^{(k)}\widetilde{x}^{(j)}\,\,\text{and} \sum_{j=k-m_k}^{k} \alpha_j^{(k)}=1.
\end{equation}
\item  $\widetilde{x}^{(k+1)}
=q(\widetilde{y}^{(k+1)})$.
\end{enumerate}
{\color{black}
This reformulation will play an important role in the proof of Theorem~\ref{thm:main1}, where we use it to prove that the sequence  $\{\|\widetilde{x}^{(k)}\|\}_{k=1}^\infty$ converges monotonically to zero.    

The equivalency of this reformulation to AA(m) can be proved 
 by showing that if $\widetilde{x}^{(j)}=(q-I)x^{(j)}$ for all $1\leq j\leq k$, then $\widetilde{x}^{(k+1)}=(q-I)x^{(k+1)}$. First, note that the $\alpha$ values in \eqref{eq:aa_optimization} and \eqref{eq:tildey_reformulation} are the same by definition. Therefore, the equivalency is derived as follows:
\begin{align}\label{eq:reformulation_linear_update}
\widetilde{x}^{(k+1)}=(q-I)x^{(k+1)} 
=&(q-I)\Big(\sum_{j=k-m_k}^{k}\alpha_j^{(k)}q(x^{(j)})\Big)
\\=&q\Big(\sum_{j=k-m_k}^{k}\alpha_j^{(k)}(q-I)x^{(j)}\Big) \nonumber =q(\widetilde{y}^{(k+1)}), 
\end{align}
where we utilized $q(x)=Wx$, $(q-I)x=(W-I)x$ (which follows from the simplification by Corollary~\ref{cor:simplify_linear}), and 
\[
\widetilde{y}^{(k+1)}=\sum_{j=k-m_k}^{k}\alpha_j^{(k)}\widetilde{x}^{(j)} 
=\sum_{j=k-m_k}^{k}\alpha_j^{(k)}\big((q-I)x^{(j)}\big).
\]

}
\subsection{Nonlinear Operator}\label{sec:nonlinear}
Moving beyond linear operators, we now assume $q$ is nonlinear and has a first-order Taylor expansion at the fixed point $x_*$, with symmetric Jacobian $\nabla q(x_*)$ having operator norm less than one. 
Let $W:=\nabla q(x_*)$, then $q(x)$ is locally approximated at $x_*$ by its first-order Taylor expansion $q(x_*)+W(x-x_*)$. 
As a result of this local approximation, the fixed-point iteration has a local $r$-linear convergence factor of $\|W\|$, and 
 we prove a modified version of AA(m) bests the local convergence factor of the fixed-point iteration.

The modified version of AA(m) we consider for this setting includes   additional bound constraints on the coefficients $\alpha_j^{(k)}$. That is, at each iteration, rather than solving \eqref{eq:aa_optimization},
the following constrained problem is solved
\begin{equation}\label{eq:AA_minimize1}
\min_{\sum_{j=k-m_k}^{k} \alpha_j^{(k)}=1, |\alpha_j^{(k)}|\leq C_0}\Big\|\sum_{j=k-m_k}^{k}\alpha_j^{(k)}\big(q(x^{(j)})-x^{(j)}\big)\Big\|,
\end{equation}
where $C_0 >0$ is a fixed constant. We refer to AA(m) with \eqref{eq:aa_optimization} replaced by \eqref{eq:AA_minimize1} as modified AA(m). 
With this slight alteration, we  
prove the local convergence of modified AA(m) is superior to that of the fixed-point iteration.

{\color{black}
\begin{theorem}[Local convergence of modified AA(m) for nonlinear operators]\label{thm:main2} Assume $x_*$ is a fixed point of $q$ and $W:=\nabla q(x_*)$ is symmetric with  $\|W\|< 1$. In addition, assume $q(x)$ and $(q-I)^{-1}(x)$ can be locally approximated by their first-order expansion at $x=x_*$. That is, there exists constants $c_1, C_1>0$ such that for all $\|x-x_*\|\leq c_1$, 
\begin{align}\label{eq:assumption_main2}
\big\|q(x)-\big(x_*+ W(x-x_*)\big)\big\|&\leq C_1\|x-x_*\|^2,\\
\big\|(q-I)^{-1}(x)-(W-I)^{-1}(x-x_*)\big\|&\leq C_1\|x-x_*\|^2.\nonumber
\end{align}
Lastly, assume $q(x)$ and $(q-I)^{-1}(x)$ are Lipschitz continuous in a neighborhood of $x_*$. That is, for all $x_1,x_2$ such that $\|x_1-x_*\|,\|x_2-x_*\|\leq c_1$
\begin{align}\label{eq:assumption_main3}
\big\|q(x_1)-q(x_2)\big\|&\leq C_1\|x_1-x_2\|,\\\,\,\big\|(q-I)^{-1}(x_1)-(q-I)^{-1}(x_2)\big\|&\leq C_1\|x_1-x_2\|.\nonumber
\end{align}
Then there exists constants $c_0', C_0'>0$ such that when $C_0$ in \eqref{eq:AA_minimize1} is chosen to be larger than $C_0'$ and the initialization $x^{(1)}$ is sufficiently close to $x_*$ in the sense that $\|x^{(1)}-x_*\|<c_0'$, then the modified AA(m) converges with local $r$-linear convergence factor no larger than $\sqrt{w_0\|W\|}$, where $w_0$ is as defined in \eqref{eq:w0}.
\end{theorem}
}
Since the {\color{black}worst-case local $r$-linear convergence factor} of the fixed-point iteration is $\|{W}\|$, Theorem \ref{thm:main2} implies that the modified AA(m) outperforms the fixed-point iteration.
The proof of Theorem \ref{thm:main2} is based on a similar argument to the proof of Theorem~\ref{thm:main1} and leverages the local first-order approximation of $q$. 

{\color{black}
\begin{remark}
It is worth noting that the modified subproblem \eqref{eq:AA_minimize1} is required for our theoretical analysis to hold; however, it is not necessary to implement in practice,  because solving the subproblem \eqref{eq:aa_optimization} for AA(m) is essentially equivalent to setting $C_0$ to be large in \eqref{eq:AA_minimize1}. 
In all of our experiments, including for the nonlinear setting, we implement AA(m). 
\end{remark}
}

\subsubsection{A Nonlinear Example: Tyler's M-estimator}\label{sec:TME}
Tyler's M-estimator (TME) \citep{Tyler1987} is a popular method for robust covariance estimation, i.e., for estimation of the covariance matrix in a way that is resistant to the influence of outliers or deviations from the assumed statistical model. 
This estimator is obtainable through fixed-point iterations, and we show one such formulation satisfies the symmetry assumption in Theorem~\ref{thm:main2}. TME provides an important instance of a nonlinear operator, and we compare the performance of AA(m) to the fixed-point iterations for TME in Section~\ref{sec:simulations}. 

Given a dataset $\{x_i\}_{i=1}^n\subset\R^p$, the standard fixed-point iteration to compute TME is  
\begin{equation}\label{eq:TME1}
\Sigma^{(k+1)}=G(\Sigma^{(k)}):=\frac{p}{n}\sum_{i=1}^n\frac{x_i x_i^T}{x_i^T\Sigma^{(k)\,-1}x_i},
\end{equation}
subject to the constraint $\text{tr}(\Sigma^{(k+1)})=p$. \cite{KentTyler1988} established the uniqueness and existence of the underlying fixed point 
of \eqref{eq:TME1} and its constraint, which defines the TME, along with proving convergence. Their theory holds under a natural geometric condition that avoids concentration on lower-dimensional subspaces, i.e., any linear subspace
of $\mathbb{R}^p$ of 
dim. 
$1 \leq d \leq p-1$ 
contains less than ${nd}/{p}$ of the data points.
\cite{FranksMoitra2020} established the linear convergence of \eqref{eq:TME1}.
TME can be interpreted as the maximum likelihood estimator of the shape matrix of
the angular central Gaussian distribution \citep{tyler1987statistical} and of the generalized elliptical
distribution \citep{frahm2}. For data i.i.d.~sampled from a continuous elliptical distribution, TME emerges as the ``most robust" covariance estimator as $n$ approaches infinity~\citep{Tyler1987}. The practical utility of TME is evident in multiple domains of applications \citep{chen+2011,FrahmJaekel2007,OllilaKoivunen2003,OllilaTyler2012}. 

An alternative iterative procedure for solving for the TME was introduced in \cite{DBLP:journals/ma/ZhangCS16}. They showed an equivalent fixed-point iteration for TME is given by 
\begin{equation}\label{eq:TME20}
w^{(k)}={F}(w^{(k-1)}),
\end{equation} 
where ${F}:\R^n\rightarrow\R^n$ and the $j$-th component of ${F}$ is given by 
\begin{equation}\label{eq:TME2M}
F_j(w)=-\log\Big(x_j^T(\frac{p}{n}\sum_{i=1}^ne^{w_i}x_i x_i^T)^{-1}x_j\Big).
\end{equation}
One can quickly verify the equivalence by letting 
\begin{equation}\label{eq:TME0}
\Sigma^{(k)}=\frac{p}{n}\sum_{i=1}^ne^{w_i^{(k)}}x_i x_i^T.
\end{equation}
Then \eqref{eq:TME2M} implies
\begin{equation}\label{eq:TME00}
w_j^{(k+1)}=-\log(x_j^T\Sigma^{(k)\,-1}x_j),
\end{equation}
and from these equations one can obtain the standard fixed-point iteration. Notably, however, though these iterations are equivalent, the standard iterative procedure fails to have a symmetric Jacobian at fixed points while the alternative procedure does. 
The next lemma proves this fact. Since our numerical results directly apply this lemma and its proof is short, we include it here.
\begin{lemma}
\label{lemma:tme_via_teng}
Let ${w}^*$  be a fixed point of \eqref{eq:TME20}. Then, $\nabla {F}(w^*)$ is symmetric.
\end{lemma}
\begin{proof}
At the fixed point $w^*$ we have
$
w_j^*=-\log\Big(x_j^T(\frac{p}{n}\sum_{i=1}^ne^{w_i^*}x_i x_i^T)^{-1}x_j\Big)
$
which implies
\begin{equation}\label{eq:help_lem1}
x_j^T\left(\frac{p}{n}\sum_{i=1}^ne^{w_i^*}x_i x_i^T\right)^{-1}x_j=e^{-w_j^*}.
\end{equation}
From \eqref{eq:TME0}, \eqref{eq:TME00}, and \eqref{eq:help_lem1} it follows
$
e^{w_j^*}=\frac{1}{x_j^T\Sigma^{*\,-1}x_j}. 
$
Therefore,
\begin{align*}
&\frac{d}{d w_i}F_j(w^*)=\frac{x_j^T\Sigma^{*\,-1}(\frac{p}{n}e^{w_i^*}x_i x_i^T)\Sigma^{*\,-1} x_j}{x_j^T\Sigma^{*\,-1}x_j}
={e^{w_j^*}x_j^T\Sigma^{*\,-1}\left(\frac{p}{n}e^{w_i^*}x_i x_i^T\right)\Sigma^{*\,-1}x_j}
\end{align*}
from which we see $\frac{d}{d w_i}F_j(w^*)=\frac{d}{d w_j}F_i(w^*)$ for all $i$ and $j$.
\end{proof}

\subsection{Discussion of Results}\label{sec:discussion}
Our work establishes an upper bound on the $r$-linear convergence factor of AA(m) which we show in many cases to be strictly less than the convergence factor of the fixed-point iteration. 
This result is a new addition to the literature because we directly show AA(m) has a better {\color{black} r-linear convergence factor} which prior works have not clearly demonstrated, e.g., \cite{Evans2020,doi:10.1137/21M1449579,pollock2021anderson,pollock2019anderson}. 

Our analysis studied the improvement of AA(m) over every pair of iterations, and we see the estimates we derived can be tight. In Section~\ref{sec:simulations}, we show in our numerical experiments our estimation in \eqref{eq:twoiterations} is usually tight for AA(1).
In practice, we observe the improvement factor varies over iterations, see Figure~\ref{fig:sim_1}, which results in a smaller r-linear convergence factor asymptotically. Additionally, we have noted the improvement factor appears to depend upon $m$; therefore, we conjecture our estimation of the r-linear convergence factor could be improved by investigating the improvement of AA over more consecutive iterations. We also expect faster convergence when $m\geq 2$, improving from our current estimation that is independent of $m$. For example, Proposition~\ref{cor:main} already shows that there are some settings where AA(m) with $m\geq 2$ has an improved convergence factor over $m=1$, and existing results, e.g., \cite{scieur2019generalized}, establish the convergence factor when $m$ is infinite, which is numerically smaller than our {\color{black}bound} of $\sqrt{w_0\|W\|}${. 
We refer the reader to the acceleration monograph by \cite{OPT-036} for a complete review. We remark that while it is possible to apply these results to analyze the convergence of Anderson acceleration with restarts after every $m$ steps~\cite[eq. (3)]{pmlr-v130-bertrand21a}, their analysis does not apply to the setting of work, which analyzes AA with a moving window of points of depth $m$.}

As a closing remark, in practice it is common to apply a damping technique in AA(m), that is, instead of \eqref{eq:aa_update}, the next iterate is updated as
\begin{align*}
x^{(k+1)}=&{\color{black}\gamma}\sum_{j=k-m_k}^{k}\alpha_j^{(k)} q(x^{(j)}) + (1-{\color{black}\gamma})\sum_{j=k-m_k}^{k}\alpha_j^{(k)} x^{(j)}
\end{align*}
for some {\color{black}$\gamma \in (0,1]$.} If this technique is applied, similar results to ours can be obtained. For example, Theorem~\ref{thm:main1} and Proposition~\ref{cor:main} hold by replacing $W$ with {\color{black}$\gamma W+(1-\gamma)I$.}

%
%
\section{Simulations}\label{sec:simulations}
We investigated our theoretical results with numerical experiments on symmetric linear and nonlinear operators in Sections~\ref{sec:lin_operator_exp} and \ref{sec:TME_experiments} respectively\footnote{Code available at: \href{https://github.com/GarnerOpt/Improved-Convergence-Factor-of-Windowed-AA.git}{https://github.com/GarnerOpt/Improved-Convergence-Factor-of-Windowed-AA.git}}.
We also compared and discussed three different variants of AA in Section~\ref{sec:compare} for TME. 
Our numerical experiments provide empirical support for our theoretical results and demonstrate the exciting effectiveness of AA(m) for TME. 

\subsection{Linear Symmetric Operators}\label{sec:lin_operator_exp}
\begin{figure}[t]
\begin{center}
\begin{minipage}[c]{0.49\linewidth}
\includegraphics[width=\linewidth]{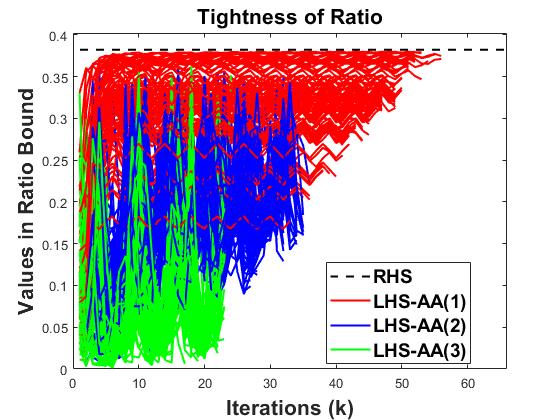}
\end{minipage}
\begin{minipage}[c]{0.49\linewidth}
\includegraphics[width=\linewidth]{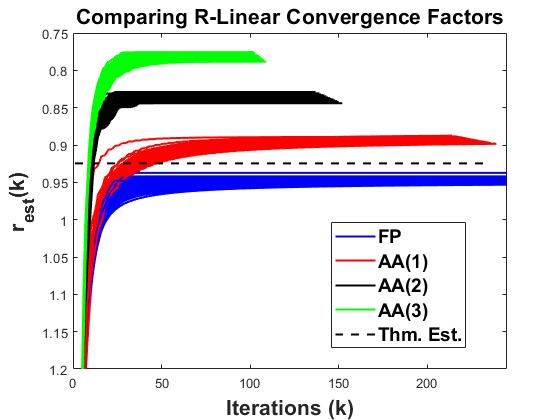}
\end{minipage}%
\caption{Displays the results of the simulations with linear operators;
{\color{black}each panel is for a different operator $q_i$.} The left panel presents the left-hand side of \eqref{eq:twoiterations} for all iterations $k\geq 2$ for one hundred random initializations of AA(m) with $q_1(x)={W}_1 x$; LHS-AA(m) denotes the value of the left-hand side of \eqref{eq:twoiterations} for the iterates of AA(m) for each random initialization while the dotted black line shows the constant right-hand side (RHS) of \eqref{eq:twoiterations}. 
The right panel displays the estimated r-linear convergence factors for the fixed-point (FP) iteration and AA(m) for $q_2(x)={W}_2x$ with one hundred random initializations of each method. The dashed line is the upper-bound on the r-linear convergence factor from Theorem \ref{thm:main1}.
}\label{fig:sim_1}
\end{center}
\end{figure}
We conducted numerical experiments on linear operators $q(x)=Wx+a$, where $W$ is symmetric and $\|W\|<1$, and verified the tightness of the upper bound on the $r$-linear convergence factor proven in Theorem \ref{thm:main1}. We present two experiments; the first verifies the tightness of \eqref{eq:twoiterations}; the second displays the upper bound on the r-linear convergence factor for AA(m).

For experiment one, we let ${W}_1 = \text{Diag}([-0.07, 0.62, -0.55,-0.6,0.15])$ and defined $q_1(x):= W_1 x$.
We computed the fixed point of $q_1$ with AA(m) for $m=1$, 2, 3 and checked inequality \eqref{eq:twoiterations} for the iterates generated by AA(m) for 100 random initializations of each algorithm. 
Each instance of AA(m) terminated once the fixed-point error,  $\|q_1(x) - x\|$, was less than $10^{-12}$. 
The random initial vectors for AA(m) had entries drawn from a standard normal. 
Figure \ref{fig:sim_1} displays the results of this experiment.  

We observe in Figure \ref{fig:sim_1} that the value of the left-hand side of \eqref{eq:twoiterations} was always bounded by the right-hand side. 
The set of iterates generated by AA(1) obtained the minimum gap of 0.0024  between the left and right-hand sides of \eqref{eq:twoiterations}.
In relative terms, the minimum gap was 0.62\% of the value of the right-hand side which strongly supports the tightness of the inequality. 

Our second experiment demonstrates the upper bound on the r-linear convergence factor for AA(m) provided by Theorem \ref{thm:main1}. 
For this experiment, we randomly generated ${W}_2 \in \mathcal{S}^{500 \times 500}$ with all but the smallest eigenvalue drawn uniformly at random between -0.9 and 0.9; the smallest eigenvalue was chosen to be -0.95. 
We applied the fixed-point iteration (FP) and AA(m) to solve $q_2(x):= {W}_2 x$ from one hundred random initializations. We used the same initialization process and termination criteria as in the first experiment.
The r-linear convergence factor was estimated as 
\begin{equation}\label{eq:r_estM}
r_{est}(k):= \max_{n\geq k} \|x_* - x^{(n)}\|^{1/n}. 
\end{equation}
The right panel in Figure \ref{fig:sim_1} displays the results of the numerical experiment. The figure shows that the predicted bound on the r-linear convergence factor is both an upper bound on the performance of AA(m) and a strict lower bound on the performance of the fixed-point iteration. 
Additionally, it shows the r-linear convergence factors for AA(m) depend on the initialization, as observed in the literature, and that the convergence factor improves as $m$ increases. 

\subsection{Nonlinear operator}\label{sec:TME_experiments}
%
%
%
\begin{figure}
\begin{center}
\begin{minipage}[c]{0.49\linewidth}
\includegraphics[width=\linewidth]{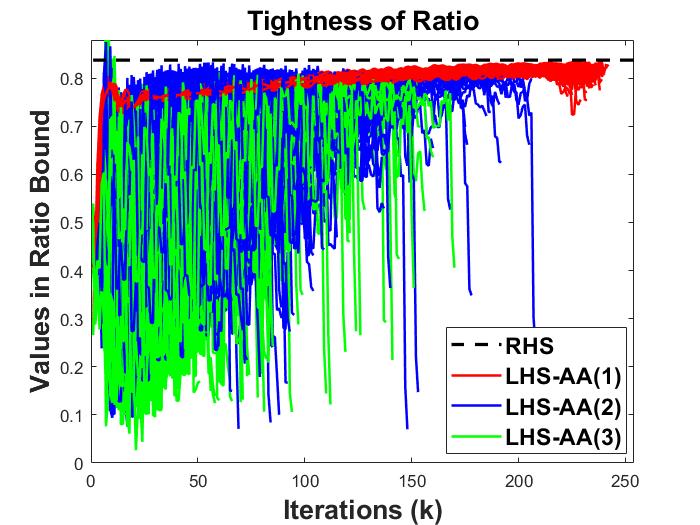}
\end{minipage}
\begin{minipage}[c]{0.49\linewidth}
\includegraphics[width=\linewidth]{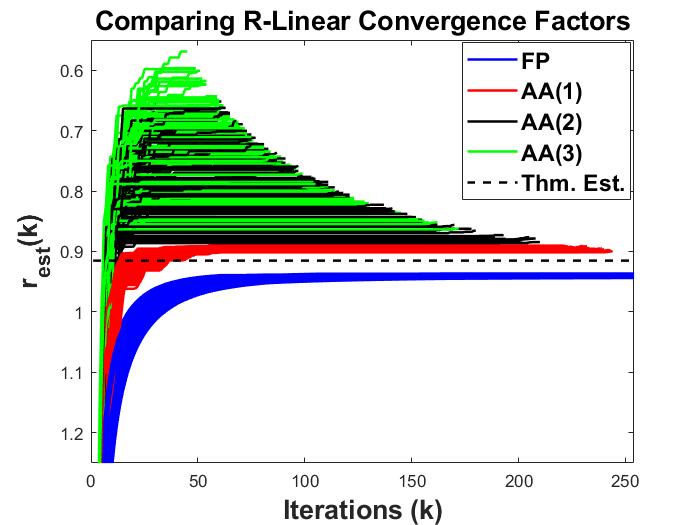}
\end{minipage}%
\caption{
Displays the results of the hundred random experiments conducted on Data Model 1.
The panel on the left displays the value of the LHS of \eqref{eq:ratio2} for iterates generated by AA(m) applied to \eqref{eq:TME2M}.
The panel on the right displays the estimated r-linear convergence {\color{black}factors} of AA(m) applied to the alternative TME fixed-point iteration. The dashed line is the upper-bound on the r-linear convergence factor given by Theorem~\ref{thm:main2} 
}\label{fig:sim_2}
\end{center}
\end{figure}
For our nonlinear operator experiments we investigated AA(m) applied to TME as described in Section \ref{sec:TME}. 
For our experiments we utilized two data models: 

\noindent{\bf Data Model 1}: Letting $(S_p)_{ij}=(0.7)^{|i-j|}$, data points were generated as $x_i = {S_p}^{1/2}{\zeta}$ where $\zeta_i \sim \mathcal{N}(0,1)$; we generated $110$ data points with $x_i \in \R^{100}$; see Section 7.1 of \cite{goes2020robust} for a full description of this model. 
 
\noindent{\bf Data Model 2}: This model is an inlier-outlier model containing 997 data points in $\R^{100}$. The data matrix ${X}$, whose columns form the dataset $\{x_i\}_{i=1}^{n}$, was formed as
\[
{X} = \begin{pmatrix}   
                \text{randn}(n_0,D)\cdot\text{randn}(D,D) \\
                \big[\text{randn}(n_1,d)/\sqrt{d}\; | \; {0} \big] 
        \end{pmatrix}^\top, 
\]
where $n_0 = 500$, $n_1 = 497$, $D=100$, $d=50$, and $\text{randn}(m,n)$ forms an $m$-by-$n$ matrix with entries drawn from a standard normal. 

We begin with numerical tests on Data Model 1. Our first experiment compared the alternative TME fixed-point iteration to AA(m) applied to \eqref{eq:TME2M} for an instance of Data Model 1 
while the second and third experiments display the superior performance of AA(m) over the standard TME fixed-point method. 

In the first experiment, each method was randomly initialized one hundred times, and all methods were terminated once the fixed-point error was less than $10^{-12}$. To fairly estimate the r-linear convergence factor of the methods, we first computed the unique solution to the TME having the trace equal to $p=100$. 
The estimates of the r-linear convergence factors for each method were then computed from \eqref{eq:r_estM} by using 
\eqref{eq:TME0} to transform each vector iterate, ${w}^{(k)}$, into a matrix which was then scaled to have trace equal to $p$. 

The reason for this estimation of the r-linear convergence factor follows from the scale-invariance of the TME, i.e., ${F}$ has multiple fixed points: if $w_*$ is a fixed point then $w_*+ {c}$, with constant vector ${c}$, is a fixed point. 
The scale-invariance also causes $W:=\nabla {F}({w}_*)$, with ${w}_*$ a fixed point of \eqref{eq:TME2M}, to have an eigenvalue of $1$ with a constant eigenvector. 
As a result, we used the second largest eigenvalue of $W$ to calculate the right-hand side of \eqref{eq:ratio2} and the upper bound on the r-linear convergence factor for AA(m). 
We observed this generated accurate bounds in our experiments.

Figure~\ref{fig:sim_2} presents the results of the experiments performed on Data Model 1.
We observe in the left panel of Figure~\ref{fig:sim_2} that the right-hand side of \eqref{eq:ratio2} serves as an upper bound on the improvement ratio. 
Our theory predicts the bound will hold for all iterates sufficiently close to a solution and the figure supports this. 
The right panel in Figure~\ref{fig:sim_2} clearly presents the upper bound on the r-linear convergence factor given by Theorem \ref{thm:main2} holds; AA(m) for $m=1$, $2$, and $3$ outperform the estimate while the fixed-point scheme given by \eqref{eq:TME2M} has an r-linear factor larger than the given estimate. 

\begin{figure}
\centering
\begin{minipage}[c]{0.47\linewidth}
\includegraphics[width=\linewidth]{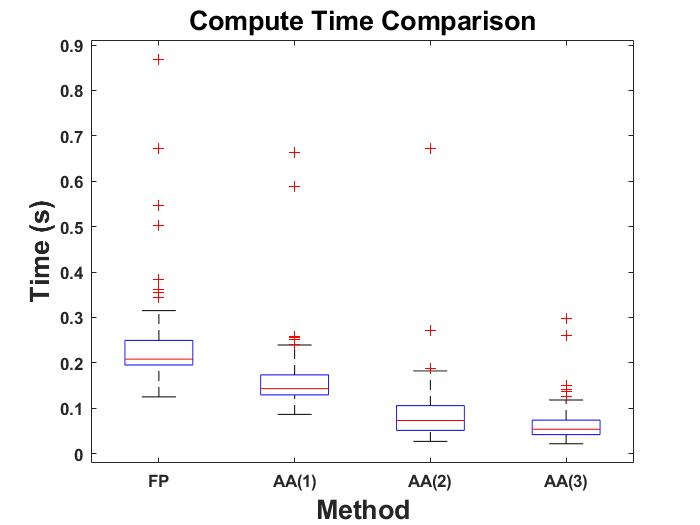}
\end{minipage}
\begin{minipage}[c]{0.47\linewidth}
\includegraphics[width=\linewidth]{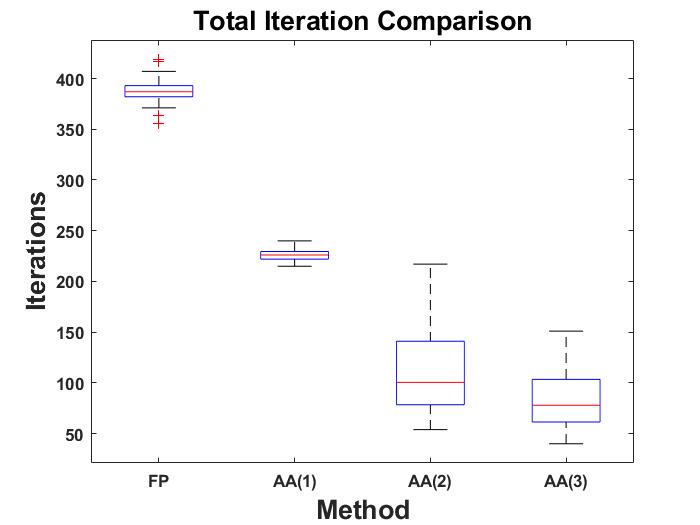}
\end{minipage}%
\caption{Displays boxplots for the computational time and total iterations required to solve TME with the standard TME fixed-point iteration (FP) and AA(m) applied to the alternative TME fixed-point iteration for Data Model 1 with  $(p,n)=(100,110)$.
}\label{fig:p100}
\end{figure}
\begin{figure}
\centering
\begin{minipage}[c]{0.47\linewidth}
\includegraphics[width=\linewidth]{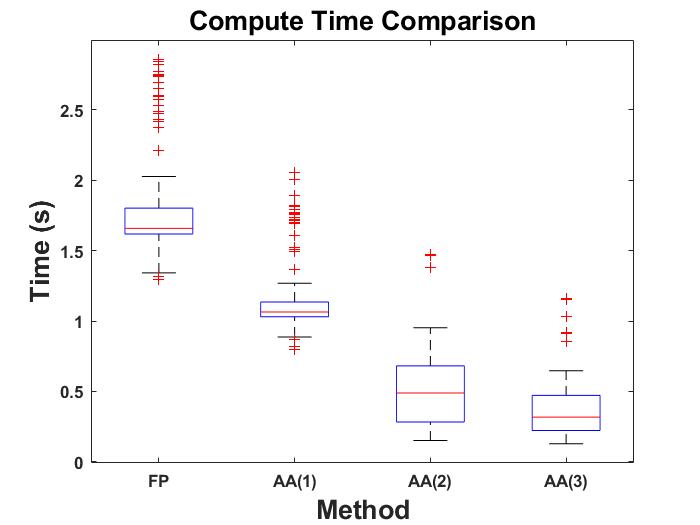}
\end{minipage}
\begin{minipage}[c]{0.47\linewidth}
\includegraphics[width=\linewidth]{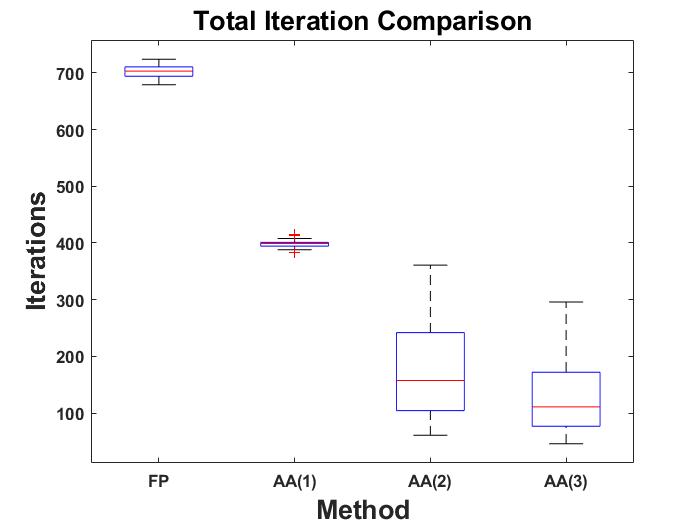}
\end{minipage}%
\caption{Displays boxplots for the computational time and total iterations required to solve TME with the standard TME fixed-point iteration (FP) and AA(m) applied to the alternative TME fixed-point iteration for Data Model 1 with $(p,n)=(200,210)$.
}\label{fig:p200}
\end{figure}
\begin{figure}
\centering
\begin{minipage}[c]{0.47\linewidth}
\includegraphics[width=\linewidth]{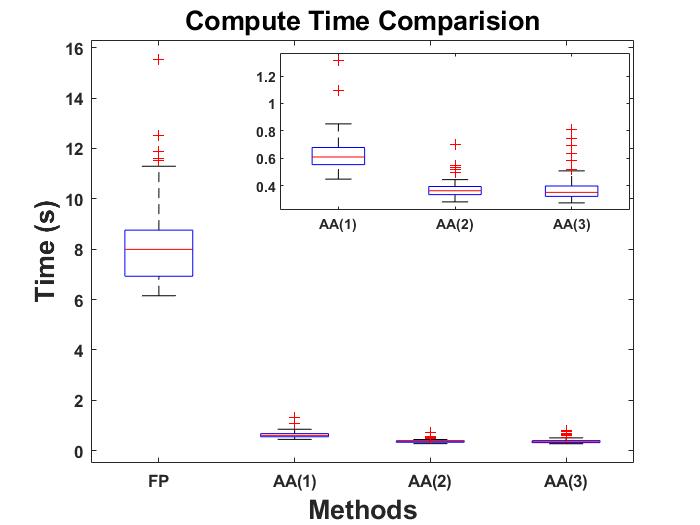}
\end{minipage}
\begin{minipage}[c]{0.47\linewidth}
\includegraphics[width=\linewidth]{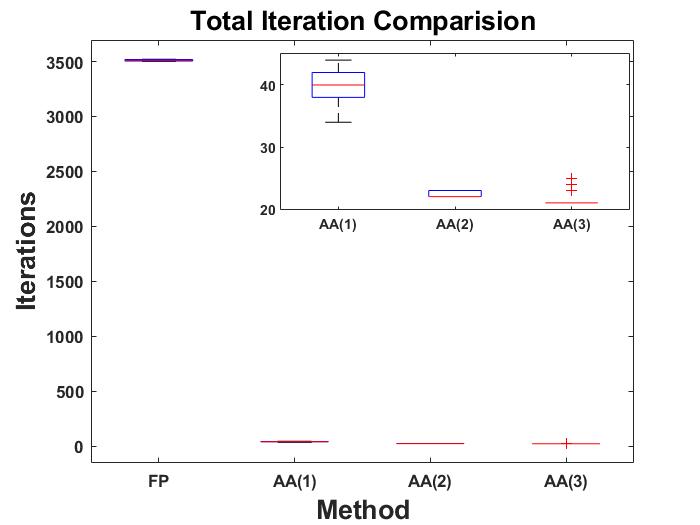}
\end{minipage}%
\caption{
Presents a box-plot of the iterations and time required for each method to compute a solution for TME with Data Model 2; FP refers to the fixed-point iteration \eqref{eq:TME1} with scaled iterate to have trace equal to $p$ while AA(m) references AA(m) applied to \eqref{eq:TME2M}. 
}\label{fig:Box_Plot} 
\end{figure}
We next present two numerical tests which compare the computational expense and performance of AA(m) and the standard TME fixed-point iteration on Data Model 1. 
We utilized the standard TME fixed-point iteration over the alternative approach when comparing computational time and cost because 
we found the fixed-point iteration given by \eqref{eq:TME1} outperformed \eqref{eq:TME2M} on these measures. 
As before, we solved ten different instantiations of Data Model 1 from ten different random initializations of our methods. The two settings of Data Model 1 we considered were $(p,n)=(100,110)$ and $(p,n)=(200,210)$. Each method ran until the fixed-point error was less than or equal to $10^{-12}$. Figures~\ref{fig:p100} and \ref{fig:p200} display the results of the numerical experiments. From the figures, we see that in all of the tests AA(m) required less iterations to compute a solution to the desired tolerance. For the setting $(p,n)=(100,110)$, AA(3) always obtained a solution faster than the fixed-point iteration while AA(2) computed a solution slower one time in all of the one hundred experiments. The results were similar for $(p,n)=(200,210)$. In this setting, AA(2) and AA(3) computed solutions faster than the fixed-point approach in all tests, and the fixed-point iteration only computed a solution faster than AA(1) in two tests. Thus, these experiments demonstrate AA provides a powerful method for solving the TME and can significantly outperform the standard fixed-point schemes. 

Data Model 2 was used in our second set of experiments. 
We solved ten different instantiations of the model with ten different random initializations of each approach for a total of one hundred implementations of each method. 
All algorithms were terminated when the fixed-point error was less than $10^{-12}$.  

The results of the numerical tests are presented in Figure \ref{fig:Box_Plot} with a set of box-plots. The figure clearly demonstrates AA(1), AA(2), and AA(3) computed solutions to TME faster than the standard fixed-point iteration for all experiments with AA(m) converging nearly ten-times faster in every test. The comparison for Data Model 2 was even more striking than for Data Model 1. 
In this inlier-outlier setting, the fixed-point method was completely outclassed by AA(m); therefore, this suggests AA(m) can prove to be a great benefit in settings where the standard fixed-point approach exhibits slow linear convergence.  

Besides investigating the computational advantage of AA(m), we further validated the strength of our theory in the non-linear setting. 
Figure~\ref{fig:extra_DM2} compares the r-linear convergence factor of AA(m) to the standard TME fixed-point iteration and displays the bound in \eqref{eq:ratio2} for the experiments displayed in Figure~\ref{fig:Box_Plot}. 
Figure~\ref{fig:Box_Plot} clearly demonstrated AA vastly decreased the compute time, and the right panel in Figure~\ref{fig:extra_DM2} gives some evidence as to why. The r-linear convergence factors of AA(2) and AA(3) in the one hundred experiments were approximately 0.4 while the r-linear convergence factor of the standard TME fixed-point iteration was about one. As discussed in Section~\ref{sec:discussion}, we note the closeness of the estimate of the r-linear convergence factor to the fixed-point method's {\color{black} convergence factor} suggests our theory can further be refined to achieve a better bound. This might be achieved by investigating the improvement of AA(m) over more iterations than currently investigated in our analysis. 

\begin{figure}
\centering
\begin{minipage}[c]{0.47\linewidth}
\includegraphics[width=\linewidth]{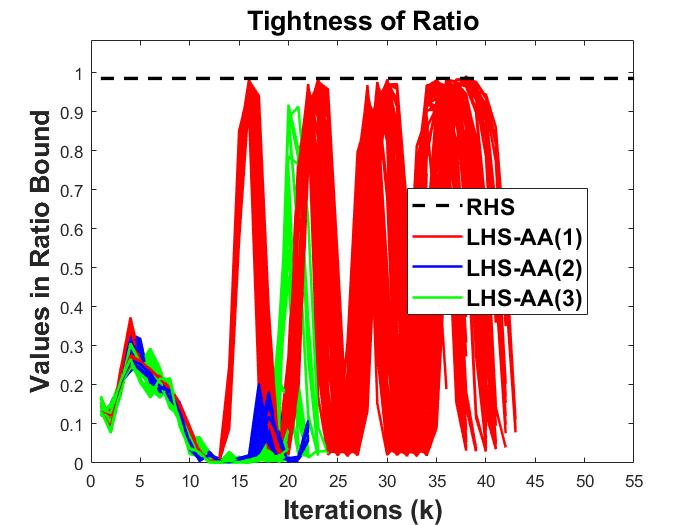}
\end{minipage}
\begin{minipage}[c]{0.47\linewidth}
\includegraphics[width=\linewidth]{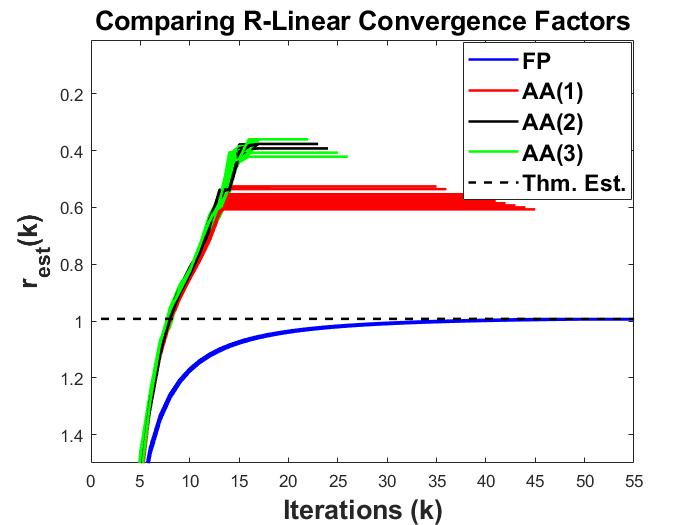}
\end{minipage}%
\caption{
Displays the results of the hundred random experiments conducted on Data Model 2.
The panel on the left displays the value of the LHS of \eqref{eq:ratio2} for iterates generated by AA(m) applied to \eqref{eq:TME2M}.
The panel on the right displays the estimated r-linear convergence {\color{black}factors} of AA(m) applied to the alternative TME fixed-point iteration. The dashed line is the upper-bound on the r-linear convergence factor given by Theorem~\ref{thm:main2}.
}\label{fig:extra_DM2}
\end{figure}

\subsection{Comparison of Anderson Acceleration Variants for TME}\label{sec:compare}
The three most popular AA variants are the following ones. 

\noindent{\bf Full-memory AA}  utilizes all prior iterates to compute the next iterate. Setting $m=\infty$ in Algorithm~\ref{alg:aa} is equivalent to this approach, so we denote it 
as AA($\infty$).

\noindent{\bf Restarting AA} operates as full-memory AA until a certain number of fixed iterations $m$ is reached at which point in time the method clears its memory and begins anew with the last iterate as a new initialization; we denote this version as restart-AA(m).

\noindent{\bf Windowed AA}, with window size $m$, is AA(m) as defined in Algorithm~\ref{alg:aa}.

For discussions on these methods and their analysis see the aforementioned references. 
Now, the different variants of AA do not always behave in manners which are easy to understand.
To illustrate this, we applied all three AA variants to {\color{black} a single instance of each of} the two TME Data models from Section~\ref{sec:TME_experiments}. The results {\color{black}from these two tests} can be viewed in Figure~\ref{fig:compare_AA}, {\color{black}where we measure the fixed-point errors of the iterates of each variant. 
Recall that the fixed-point error at $x \in \R^n$ for an operator $q:\R^n \rightarrow \R^n$ is $\|q(x)-x\|$.}
\begin{figure}
\centering
\begin{minipage}[c]{0.47\linewidth}
\includegraphics[width=\linewidth]{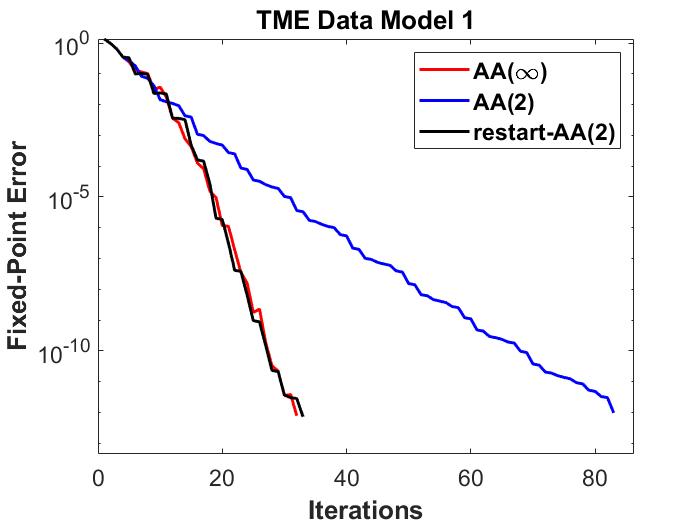}
\end{minipage}
\begin{minipage}[c]{0.47\linewidth}
\includegraphics[width=\linewidth]{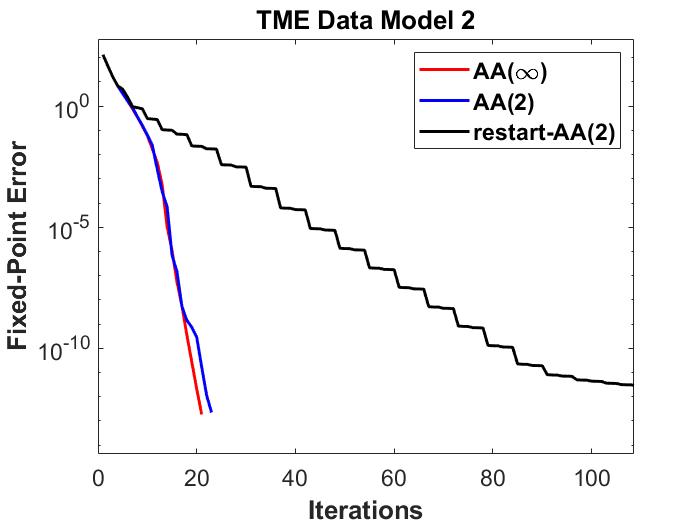}
\end{minipage}%
\caption{Displays the results of comparing full-memory AA, restarting AA, and windowed AA on the TME problem with Data Models 1 and 2. The fixed-point error is the norm of the difference of the consecutive iterates of each method. 
}\label{fig:compare_AA}
\end{figure}
For Data Model 1, restarting AA with $m=2$ performed just as {\color{black}well} as full-memory AA and was significantly better than windowed AA with $m=2$; however, the roles were reversed when the methods were applied to the inlier-outlier data model, Data Model 2. 
In this setting, windowed AA was clearly better while restarting AA appeared to alternate between different {\color{black} convergence factors}. 
We do not presently have a clear explanation for the difference in the behavior between the two methods though we suspect the phenomenon is highly problem dependent. 
This could be a fruitful direction to explore in the future.  

\section{Technical Proofs}\label{sec:proofs}

\subsection{Proof of Theorem~\ref{thm:main1}}

The proof will be divided into two parts. In the first part, we establish \eqref{eq:twoiterations} and that the convergence factor of AA(m) is bounded above by $\sqrt{w_0\|W\|}$. In the second part, we demonstrate that $w_0 \leq \|W\|$, and equality holds if and only if $\lambda_{\max}(W) = -\lambda_{\min}(W)$.  The two parts of the proof are completely independent of each other.
 
\textbf{Part I: proof of \eqref{eq:twoiterations} and the bound on the convergence factor of AA(m)} As described in Section~\ref{sec:reformulate_AA} and Corollary~\ref{cor:simplify_linear}, we may assume without loss of generality ${a}={0}$, $q(x)=Wx$,  ${x}_* = {0}$, and reformulate AA(m) as follows: 
\begin{enumerate}
\item Let $\widetilde{y}^{(k+1)}$ be the point on the affine subspace spanned by $\widetilde{x}^{(k-m_k)}, \cdots, \; \widetilde{x}^{(k)}$ with the smallest norm.

\item  $\widetilde{x}^{(k+1)}
=q(\widetilde{y}^{(k+1)})$.
\end{enumerate}
Here $m_k:=\min(m,k)$ and we assumed in our derivation of the reformulation $\widetilde{x}^{(k)}=q(x^{(k)})-x^{(k)}$. We first verify \eqref{eq:twoiterations} and begin by noting that \eqref{eq:twoiterations} is equivalent to
\[
\frac{\|\widetilde{x}^{(k+1)}\|}{\|\widetilde{x}^{(k-1)}\|}\leq w_0 \|W\|.
\]
Since $\widetilde{x}^{(k+1)} 
=q(\widetilde{y}^{(k+1)})$ by Step 2 of the reformulation above, it follows thats
$
\|\widetilde{x}^{(k+1)}\|\cdot \|W\|^{-1} \leq \|\widetilde{y}^{(k+1)}\|.
$
Therefore, to prove \eqref{eq:twoiterations}, it is sufficient to show 
 \begin{equation}\label{eq:twoiterations1}
\frac{\|\widetilde{y}^{(k+1)}\|}{\|\widetilde{x}^{(k-1)}\|}\leq w_0.
 \end{equation}
We now prove that  \eqref{eq:twoiterations1} holds. Let $\widehat{y}^{(k+1)}$ be the projection of the origin to the line connecting $\widetilde{x}^{(k-1)}$ and $\widetilde{x}^{(k)}$, thereby making it an affine combination of $\widetilde{x}^{(k-1)}$ and $\widetilde{x}^{(k)}$. By the definition of $\widetilde{y}^{(k+1)}$ in Step 1 of the reformulation, we have
\begin{equation}\label{eq:twoiterations2}
\|\widehat{y}^{(k+1)}\|\geq \|\widetilde{y}^{(k+1)}\|.
\end{equation}
Combine \eqref{eq:twoiterations2} with
\begin{equation*}
\frac{\|\widehat{y}^{(k+1)}\|}{\|\widetilde{x}^{(k-1)}\|}=\sin\angle(O \widetilde{x}^{(k-1)} \widehat{y}^{(k+1)})=\sin\angle(O\widetilde{x}^{(k-1)}\widetilde{x}^{(k)}),
\end{equation*}
it follows that to prove   \eqref{eq:twoiterations1}, it is sufficient to show \begin{equation}\label{eq:twoiterations3}\sin\angle(O\widetilde{x}^{(k-1)}\widetilde{x}^{(k)}) \leq w_0.\end{equation} 
To prove \eqref{eq:twoiterations3}, we introduce the following lemma:    
\begin{lemma}\label{lemma:bound}
For any symmetric matrix $W$, $Wx$ lies in the ball centered at\\ $(\lambda_{\max}(W)+\lambda_{\min}(W))x/2$ with radius $(\lambda_{\max}(W)-\lambda_{\min}(W))\|x\|/2$.
\end{lemma}
\begin{proof}
Given any $W \in \mathcal{S}^{n\times n}$ and $x\in \R^n$, we see
\begin{align}
\Big\|Wx - \frac{1}{2}\big(\lambda_{\max}(W) +\lambda_{\min}(W)\big)x \Big\| &\leq \Big\| W - \frac{1}{2}\big(\lambda_{\max}(W) +\lambda_{\min}(W)\big){I}\Big\| \cdot \|x\|\nonumber \\
&= \frac{1}{2}\big(\lambda_{\max}(W) -\lambda_{\min}(W)\big)\|x\|. \label{eq:radius_argument}
\end{align}
\end{proof}

\begin{figure}[t]
\begin{center}
\includegraphics[scale=0.45]{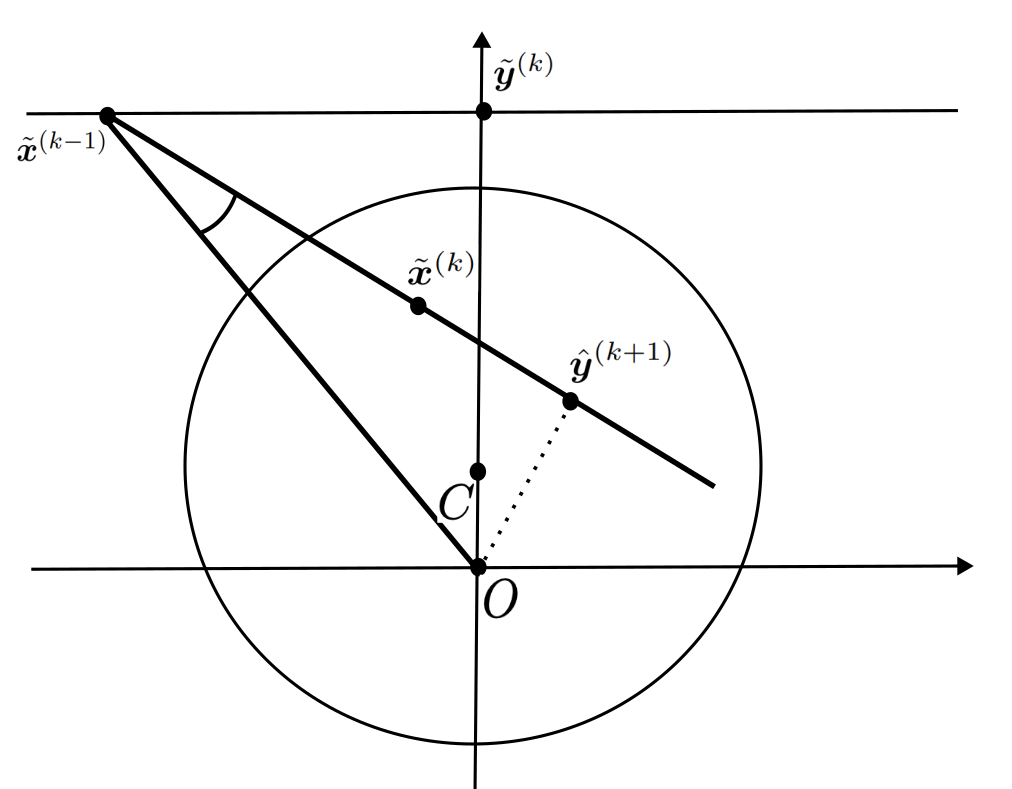}\caption{A visualization of $\widetilde{x}^{(k-1)}$, $\widetilde{x}^{(k)}$, $\widetilde{y}^{(k)}$, and $\widehat{y}^{(k+1)}$, as described in Part I of the proof of Theorem~\ref{thm:main1}. The ball with center $C$ represents the region defined by Lemma~\ref{lemma:bound}, showing the possible values of $W \widetilde{y}^{(k)}$.
}\label{fig:visualization_proof}
\end{center}
\end{figure}

{
We now utilize a geometric argument. First, we note that, unitary transformation and scaling does not change the performance of AA(m). Therefore, without loss of generality, we may assume  $\widetilde{y}^{(k)} = [0,1,0,\hdots,0]$. Then, following the definition of $\widetilde{y}^{(k)}$ in \eqref{eq:tildey_reformulation}, $\widetilde{x}^{(k-1)}-\widetilde{y}^{(k)}$ is perpendicular to $\widetilde{y}^{(k)}$, and $\widetilde{x}^{(k-1)}$ under our chosen coordinate system lies on the hyperplane through the point $(0,1,0,\hdots,0)$ with a normal vector being $(0,1,0,\hdots,0)$. Therefore, with an additional unitary transformation we may also assume $\widetilde{x}^{(k-1)}-\widetilde{y}^{(k)}$ is parallel to the first standard basis vector in $\R^n$. A geometric representation of our transformed problem, presenting the relationship between $\widetilde{x}^{(k-1)}$, $\widetilde{x}^{(k)}$, $\widetilde{y}^{(k)}$, and $\widehat{y}^{(k+1)}$, is given in Figure~\ref{fig:visualization_proof}. 

Using Lemma~\ref{lemma:bound}, we denote in Figure~\ref{fig:visualization_proof} all the possible values of $\widetilde{x}^{(k)}=W \widetilde{y}^{(k)}$, which is given by all the points contained in the ball centered at 
$C=(\lambda_{\max}(W)+\lambda_{\min}(W))\widetilde{y}^{(k)}/2$ with radius
\begin{equation}\label{eq:radius}
r = (\lambda_{\max}(W)-\lambda_{\min}(W))\|\widetilde{y}^{(k)}\|/2=(\lambda_{\max}(W)-\lambda_{\min}(W))/2. 
\end{equation}
}We remark that $\widetilde{x}^{(k)}$ does not need to lie on the plane spanned by $\widetilde{x}^{(k-1)}$ and $\widetilde{y}^{(k)}$, as illustrated in Figure~\ref{fig:visualization_proof}, and the remainder of the proof remains valid even if $\widetilde{x}^{(k)}$ is outside this plane.

Recall that our aim is to bound  $\angle(O\widetilde{x}^{(k-1)}\widetilde{x}^{(k)})$ and prove \eqref{eq:twoiterations3}.
{
Note that 
\begin{align}\label{eq:geometric1}
\angle(O\widetilde{x}^{(k-1)}\widetilde{x}^{(k)}) &\leq \angle(O\widetilde{x}^{(k-1)}C)+\angle(C\widetilde{x}^{(k-1)}\widetilde{x}^{(k)}) \nonumber \\
&\leq \angle(O\widetilde{x}^{(k-1)}C)+\sin^{-1}\left(\frac{r}{\|C - \widetilde{x}^{(k-1)}\|}\right),
\end{align} 
where the first inequality follows from the triangle inequality and the second inequality follows from 

\begin{align}
\sin\angle(C\widetilde{x}^{(k-1)}\widetilde{x}^{(k)})= \frac{\dist(C, \text{line connecting $\widetilde{x}^{(k-1)}$ and $\widetilde{x}^{(k)}$})}{\|C-\widetilde{x}^{(k-1)}\|}\leq \frac{r}{\|C-\widetilde{x}^{(k-1)}\|}.
\label{eq:geometric2}\end{align}
Now the estimation of the upper bound of the RHS of \eqref{eq:geometric1} follows from analyzing the two parts separately. 
Recall by our prior assumption it follows $\|\widetilde{y}^{(k)}\|=1$ and assume in addition $\|\widetilde{y}^{(k)}-\widetilde{x}^{(k-1)}\|=1/\delta$, then we have 
\begin{align}
\frac{r}{\|C - \widetilde{x}^{(k-1)}\|}= \frac{|\lambda_{\max}(W)-\lambda_{\min}(W)|\delta}{\sqrt{4+\big(2-\lambda_{\max}(W)-\lambda_{\min}(W)\big)^2\delta^2}}
\label{eq:geometric3}\end{align}
and 
\begin{align}
\nonumber\angle(O\widetilde{x}^{(k-1)}C)=&\bigg|\angle(O\widetilde{x}^{(k-1)}\widetilde{y}^{(k)})-\angle(\widetilde{y}^{(k)}\widetilde{x}^{(k-1)}C)\bigg|\\
=&\bigg|\tan^{-1}(\delta)-\tan^{-1}\Big(\Big(1-\frac{\lambda_{\max}(W)+\lambda_{\min}(W)}{2}\Big)\delta\Big)\bigg|. 
\label{eq:geometric4}\end{align}
Combining 
\eqref{eq:geometric1}, \eqref{eq:geometric3}, \eqref{eq:geometric4}, and the definition of $w_0$ in \eqref{eq:w0}, we proved \eqref{eq:twoiterations3}. As discussed previously, it follows that
\eqref{eq:twoiterations1} and \eqref{eq:twoiterations} are  established.

}




Since $\|W\|<1$, the operator 
$q$ is a contraction, which implies, by the work of \cite{Toth2015}, that Anderson acceleration converges to the unique  fixed-point $x_*=0$.
Using \eqref{eq:twoiterations} we then see the r-linear convergence factor is bounded by $\sqrt{w_0\|W\|}$. 
Observe that, from \eqref{eq:twoiterations}, it follows
\begin{multline}\label{eq:r_linear_factor}
\|\widetilde{x}^{(k)}\| \leq w_0 \|W\| \|\widetilde{x}^{(k-2)}\| \leq \left(w_0 \|W\|\right)^2 \|\widetilde{x}^{(k-4)}\| \\ \leq \hdots \leq \left( w_0 \|W\| \right)^{\lfloor (k-2)/2 \rfloor+1} \max\left(\|\widetilde{x}^{(0)}\|,\|\widetilde{x}^{(1)}\|  \right), 
\end{multline}
and so
\[
\limsup_{k\rightarrow \infty} \|\widetilde{x}^{(k)}\|^{1/k} \leq \lim_{k\rightarrow \infty} \left( \left( w_0 \|W\| \right)^{\lfloor (k-2)/2 \rfloor+1} \max\left(\|\widetilde{x}^{(0)}\|,\|\widetilde{x}^{(1)}\|  \right)  \right)^{1/k} \hspace{-0.1in} = \sqrt{w_0 \|W\|}. 
\]
\textbf{Part II: proof of $w_0\leq \|W\|$ and the conditions under which $w_0 = \|W\|$.
} {To establish $w_0\leq \|W\|$ and the strict improvement on the convergence factor, we use an alternative geometric interpretation of $w_0$, which is independent of AA(m). 
Lemma~\ref{lemma:w0} presents the geometric interpretation and a visualization is provided in Figure~\ref{fig:visualization_proof2};
the proof is deferred to Section~\ref{sec:proof_lemma_w0}.  

\begin{lemma}\label{lemma:w0}
Consider a circle centered at $C=(0,[\lambda_{\max}(W)+\lambda_{\min}(W)]/{2})$ with radius $[\lambda_{\max}(W)-\lambda_{\min}(W)]/{2}$, and let $O$ denote the origin. Then $w_0$, as defined in \eqref{eq:w0}, is the maximum value of $\sin\angle(OPQ)$ as $P$ traverses along the line $y=1$ and $Q$ traverses along the circle.
\end{lemma}

Using the geometric understanding in Lemma~\ref{lemma:w0}, we have 
\[
\sin \angle(OPQ)=\frac{\dist(O,L)}{|OP|},
\]
where $L$ denotes the line connecting $P$ and $Q$. The numerator and the denominator of its RHS are bounded by 
\begin{align}\label{eq:equality1}|OP|\geq& |OR|=1,\\\label{eq:equality2}\dist(O, L)\leq& |OC|+|CQ| =  \|W\|.\end{align} It follows that $\sin \angle(OPQ)\leq \|W\|$. Since $w_0$ is the largest possible $\sin \angle(OPQ)$, we have $w_0\leq \|W\|$. 

From the argument above,  $w_0 = \|W\|$ only when \eqref{eq:equality1} and \eqref{eq:equality2} are equalities, that is, $P=R$ and $O,C,Q$ are collinear. If $C\neq O$, then the collinearity of $O,C,Q$ implies that $Q$ must lie on the $y$-axis in Figure~\ref{fig:visualization_proof2}; however, this would violate the required tangency of $PQ$ with the circle. As a result, $w_0 = \|W\|$  will only occur when $O=C$, i.e., $\lambda_{\min}(W)+\lambda_{\max}(W)=0$. 

On the other hand, when $\lambda_{\min}(W)+\lambda_{\max}(W)=0$, we have $C=O$ in Figure~\ref{fig:visualization_proof2} and the circle has a radius of $\|W\|$. Let $P=R$, it follows that $\sin \angle(OPQ)=\|W\|$. Since $w_0\leq \|W\|$ and Lemma~\ref{lemma:w0} implies $\angle(OPQ)\leq w_0$, we proved $w_0=\|W\|$.

\begin{figure}[t]
\begin{center}
\includegraphics[scale=0.35]{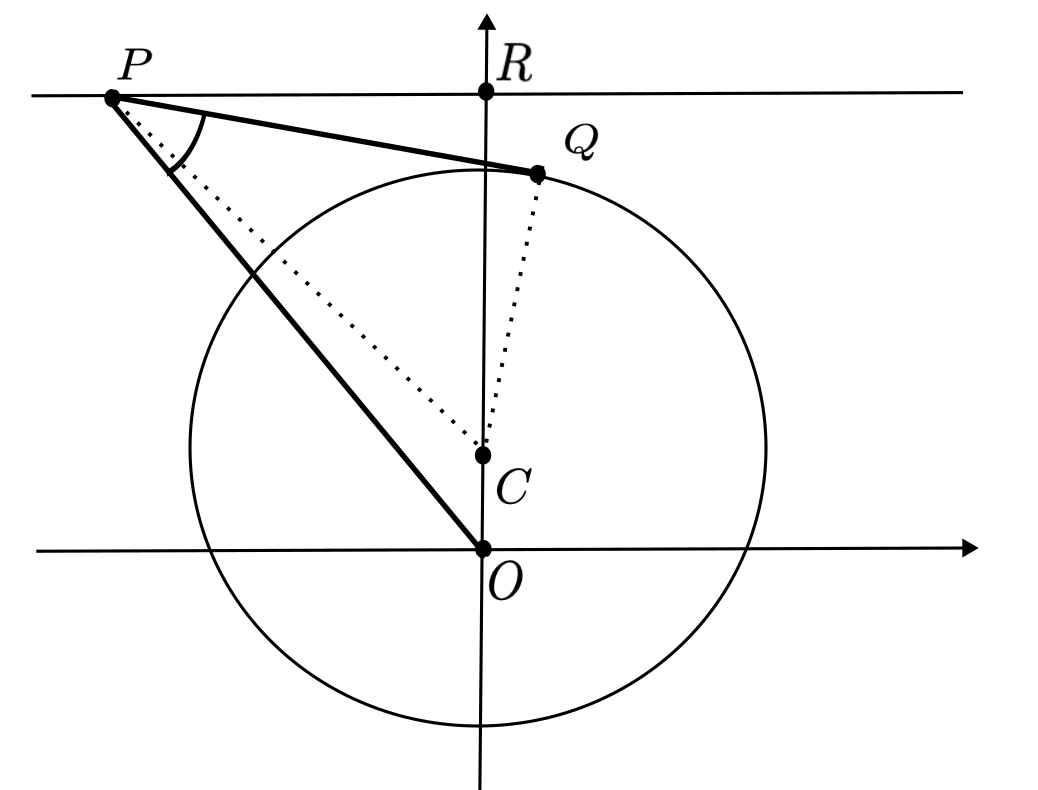}\caption{Geometric representation of $w_0$ as the sine of the largest possible angle $\angle(OPQ)$ formed as  $P$ traverses along the line $y=1$ and $Q$ traverses along the circle.}\label{fig:visualization_proof2}
\end{center}
\end{figure}

%
\begin{remark}
We emphasize that, although the visualizations in Figures~\ref{fig:visualization_proof} and \ref{fig:visualization_proof2} have some similarities, they serve different purposes for distinct arguments. Figure~\ref{fig:visualization_proof} is used in part I of the proof and aims to establish \eqref{eq:twoiterations} by showing that $w_0$ is an upper bound of $\sin\angle(O\widetilde{x}^{(k-1)}\widetilde{x}^{(k)})$. 
 Figure~\ref{fig:visualization_proof2}, on the other hand, is used in part II of the proof and provides an alternative understanding of $w_0$ from a geometric perspective, independent of Anderson acceleration. This alternative perspective is easier to understand, enabling us to demonstrate a strict improvement in the convergence factor.
\end{remark}
\subsubsection{Proof of Lemma~\ref{lemma:w0}}\label{sec:proof_lemma_w0}
We prove Lemma~\ref{lemma:w0} in two parts. In Part 1, we show that, for any $P$ and $Q$ defined in Lemma~\ref{lemma:w0}, $w_0 \geq \sin\angle(OPQ)$. In Part 2, we prove the existence of $P$ and $Q$ defined in Lemma~\ref{lemma:w0} such that $w_0 = \sin\angle(OPQ)$. The combination of these two parts establishes Lemma~\ref{lemma:w0}.

\vspace{0.1in}
\textbf{Part 1: ($w_0\geq \sin\angle(OPQ)$}). We first observe that the maximum $\angle (CPQ)$ is achieved when $PQ$ is tangent to the circle. Consequently, to prove that $w_0\geq \sin\angle(OPQ)$, it is sufficient to consider the case where $PQ$ is tangent to the circle.

Assume $|PQ|=1/\delta$, then we have
\[
\sin\angle (CPQ)\leq \frac{|CQ|}{|CP|}= \!\frac{|\lambda_{\max}(W)-\lambda_{\min}(W)|\delta}{\sqrt{\!4+\!\big(2\!-\!\lambda_{\max}(\!W\!)\!-\!\lambda_{\min}(W\!)\big)^2\!\delta^2}}
\]
and 
\[
\angle (OPC) = |\angle (OPR)-\angle (CPR)|= \Big|\tan^{-1}(\delta)-\tan^{-1}\Big(\big(1-\frac{\lambda_{\max}(W)+\lambda_{\min}(W)}{2}\big)\delta\Big)\Big|.
\]
Combining these inequalities with $\angle(OPQ)\leq \angle(OPC)+\angle(CPQ)$, which holds by the triangle inequality and our convention that angles are positive, we have $w_0\geq \sin\angle(OPQ)$.

\vspace{0.1in}
\textbf{Part 2: (Existence of $P$ and $Q$ such that $w_0 = \sin\angle(OPQ)$)}.
Assuming the supremum of the RHS of \eqref{eq:w0} is achieved at $\delta_*$, choose $P$ such that its distance to the $y$-axis is $1/\delta_*$. In addition, $Q$ is chosen such that $PQ$ is tangent to the circle and $C$ lies inside the cone formed by the vectors $\vec{PQ}$ and $\vec{PO}$. Then, all the equalities in Part 1 hold with $\delta=\delta_*$, and $\angle(OPQ)= \angle(OPC)+\angle(CPQ)$. Therefore,   
\begin{align*}
&\sin\angle(OPQ)\!=\! \sin(\angle(OPC)\!+\!\angle(CPQ))\!=\!\sin\!\Bigg[\!
\sin^{-1}\!\!\Big(\!\frac{|\lambda_{\max}(W)-\lambda_{\min}(W)|\delta_*}{\sqrt{\!4+\!\big(2\!-\!\lambda_{\max}(\!W\!)\!-\!\lambda_{\min}(W\!)\big)^2\!\delta_*^2}}\!\Big)\! \\
&+ \Big|\tan^{-1}(\delta_*)-\tan^{-1}\Big(\big(1-\frac{\lambda_{\max}(W)+\lambda_{\min}(W)}{2}\big)\delta_*\Big)\Big|\Bigg]=w_0.
\end{align*}
Thus, $w_0=\sin\angle(OPQ)$ when $P$ and $Q$ are chosen in this way. 

If $\delta_*$ is undefined, then the supremum is not attained at a finite $\delta$, meaning the value of $w_0$ is obtained when  $\delta\rightarrow\infty$, that is,  
$w_0=\frac{|\lambda_{\max}(W)-\lambda_{\min}(W)|}{|2\!-\!\lambda_{\max}(\!W\!)\!-\!\lambda_{\min}(W\!)|}$. We choose $P=R$ and choose $Q$ such that $PQ$ is tangent to the circle, then $\sin\angle(OPQ)=\frac{|CQ|}{|CR|}=\frac{r}{1-\frac{\lambda_{\max}(W)+\lambda_{\min}(W)}{2}}=w_0$.

\subsection{Proof of Proposition~\ref{cor:special}} Let $W = w{I}$ with $|w|<1$. We define $\widetilde{y}^{(k)}$ and $\widetilde{x}^{(k)}$ as discussed in the proof of Theorem~\ref{thm:main1} and assume without loss of generality that $\widetilde{y}^{(k)} = [0,1,0,\hdots,0]$ and that $\widetilde{x}^{(k-1)}$ lies on the hyperplane through the point $(0,1,0,\hdots,0)$ with normal vector $(0,1,0,\hdots,0)$. Given the geometric meaning of $w_0$ in Lemma~\ref{lemma:w0}, we know that the circle in Figure~\ref{fig:visualization_proof} has radius $0$ and is reduced to the  point $C=w \widetilde{y}^{(k)}$. Then Lemma~\ref{lemma:w0} states that  $w_0$ is equivalent the largest possible $\sin\angle(OPC)$, where $P$ lies on the line $y=1$ and $O$ is the origin. 

Let  $R=(0,1)$ and $|PR| = 1/\delta$ with $\delta>0$. 
Then, 
\[
\angle(OPC)=\angle(OPR)-\angle(CPR)=\tan^{-1}(\delta)-\tan^{-1}\left((1-\|W\|)\delta\right).
\]
Taking the derivative of this  function, its maximum is achieved when 
\[
\frac{1}{1+\delta^2}=\frac{(1-\|W\|)}{1+(1-\|W\|)^2\delta^2}\;
\text{, that is, }\; \delta=\frac{1}{\sqrt{1-\|W\|}}.
\]
Applying the trigonometric identities: $\sin(\theta_1 -\theta_2) = \sin(\theta_1)\cos(\theta_2) - \cos(\theta_1)\sin(\theta_2)$, $\sin\tan^{-1}(x)=\frac{x}{\sqrt{x^2+1}}$, and $\cos\tan^{-1}(x)=\frac{1}{\sqrt{x^2+1}}$, we see
\begin{align}
\sin(\angle(OPC))&=\sin\left(\tan^{-1}\frac{1}{\sqrt{1-\|W\|}}\right)\cos(\tan^{-1}{\sqrt{1-\|W\|}}) \nonumber\\
&\;\;\;\;\;\;\;-\cos\left(\tan^{-1}\frac{1}{\sqrt{1-\|W\|}}\right)\sin(\tan^{-1}{\sqrt{1-\|W\|}}) \nonumber \\ 
&=\frac{1}{\sqrt{2-\|W\|}}\frac{1}{\sqrt{2-\|W\|}}-\frac{\sqrt{1-\|W\|}}{\sqrt{2-\|W\|}}\frac{\sqrt{1-\|W\|}}{\sqrt{2-\|W\|}}=\frac{\|W\|}{2-\|W\|}. \nonumber 
\end{align}
Thus, $w_0 =\|W\|/(2-\|W\|)$. 
\subsection{Proof of Proposition~\ref{cor:main}}

We split the proof into two parts. In the first part, we assume $m\geq 2$ and aim to prove that r-linear convergence factor of AA(m) is strictly smaller than $\|W\|$.
Following Theorem~\ref{thm:main1}, we only need to consider the case $\lambda_{\max}(W) = -\lambda_{\min}(W)$. We will prove that AA(m) has a strictly small convergence factor than $\|W\|$ 
by showing the existence of a constant $c<1$ such that, for all $k\geq 1$,
\begin{equation}\label{eq:contradiction1}
\|\widetilde{x}^{(k+2)}\| \leq c \|W\|^3\|\widetilde{x}^{(k-1)}\|,
\end{equation}
{which would imply a convergence factor of at most $c^{1/3}\|W\|$ by an argument similar to the one given in \eqref{eq:r_linear_factor}.

 If \eqref{eq:contradiction1} does not hold for all iterations $k\geq 1$, then, 
given any $c \in (0,1)$ there must exist an index $k$ such
that 
\begin{equation}\label{eq:contradiction10}
\|\widetilde{x}^{(k+2)}\| > c \|W\|^3\|\widetilde{x}^{(k-1)}\|. 
\end{equation}
We prove \eqref{eq:contradiction1} by showing \eqref{eq:contradiction10} must lead to a contradiction.

Following the definitions of $\widetilde{y}^{(k)}$ and $\widetilde{x}^{(k)}$ in Section~\ref{sec:reformulate_AA}, we have
$\|\widetilde{y}^{(k)}\|\leq \|\widetilde{x}^{(k-1)}\|$,\\
$
\|\widetilde{x}^{(k)}\|=\|W\widetilde{y}^{(k)}\|\leq \|W\|\|\widetilde{y}^{(k)}\|\leq  \|W\|\|\widetilde{x}^{(k-1)}\|, 
$
and 
\begin{align}\label{eq:simplerate_argument}
&\max\Big(\frac{\|\widetilde{y}^{(k)}\|}{\|\widetilde{x}^{(k-1)}\|},\frac{\|\widetilde{x}^{(k)}\|}{\|W\|\|\widetilde{y}^{(k)}\|},
\frac{\|\widetilde{y}^{(k+1)}\|}{\|\widetilde{x}^{(k)}\|}, \\
&\hspace{0.95in}\frac{\|\widetilde{x}^{(k+1)}\|}{\|W\|\|\widetilde{y}^{(k+1)}\|}, 
\frac{\|\widetilde{y}^{(k+2)}\|}{\|\widetilde{x}^{(k+1)}\|},
\frac{\|\widetilde{x}^{(k+2)}\|}{\|W\|\|\widetilde{y}^{(k+2)}\|}\Big)\leq 1. \nonumber 
\end{align}
Since \eqref{eq:contradiction10} implies that the product of these six terms is larger than $c$, each of the six terms is larger than $c$. 
That is, 
\begin{align}\label{eq:simplerate_argument2}
&\min\Big(\frac{\|\widetilde{y}^{(k)}\|}{\|\widetilde{x}^{(k-1)}\|},\frac{\|\widetilde{x}^{(k)}\|}{\|W\|\|\widetilde{y}^{(k)}\|},
\frac{\|\widetilde{y}^{(k+1)}\|}{\|\widetilde{x}^{(k)}\|},\\
&\hspace{0.95in}\frac{\|\widetilde{x}^{(k+1)}\|}{\|W\|\|\widetilde{y}^{(k+1)}\|},\frac{\|\widetilde{y}^{(k+2)}\|}{\|\widetilde{x}^{(k+1)}\|},
\frac{\|\widetilde{x}^{(k+2)}\|}{\|W\|\|\widetilde{y}^{(k+2)}\|}\Big)>c, \nonumber 
\end{align}
which implies for $i=0,1,2$, that
\begin{equation}\label{eq:contradiction2}
 \|\widetilde{y}^{(k+i)}\|> c\|\widetilde{x}^{(k-1+i)}\|,\,\,\,
\|W\widetilde{y}^{(k+i)}\|> c \|W\|\|\widetilde{y}^{(k+i)}\|.
\end{equation}
The first inequality in \eqref{eq:contradiction2} implies that
\begin{equation}\label{eq:nearx}
\frac{\|\widetilde{x}^{(k-1+i)}-\widetilde{y}^{(k+i)}\|}{\|\widetilde{y}^{(k+i)}\|}< \frac{\sqrt{1-c^2}}{c},
\end{equation}
which follows by the Pythagorean Theorem and the fact $\widetilde{x}^{(k-1+i)}-\widetilde{y}^{(k+i)}\perp \widetilde{y}^{(k+i)}$ from the definition of $\widetilde{y}^{(k+i)}$ in \eqref{eq:tildey_reformulation}.

The second inequality in \eqref{eq:contradiction2} implies that 
there exists $\widetilde{W}$ such that $\|\widetilde{W}\|=\|W\|$, $\widetilde{W}^2 =\|W\|^2I$, and 
\begin{equation}\label{eq:neary}
\frac{\|\widetilde{x}^{(k+i)}-\widetilde{W}\widetilde{y}^{(k+i)}\|}{\|{\widetilde{y}}^{(k+i)}\|}<2\|W\|\sqrt{\frac{1-c^2}{1-\beta^2}},
\end{equation}
where $\beta=\max_{i: \lambda_i(W)\neq \pm\|W\|}|\lambda_i(W)|/\|W\|$, and we set $\beta=0$ if all eigenvalues of $W$ are $\pm\|W\|$. We defer the proof of \eqref{eq:neary} to Section~\ref{sec:proof_neary}.

Combining \eqref{eq:nearx}, \eqref{eq:neary}, and $\|\widetilde{W}\|=\|W\|\leq 1$, we have
\begin{align*}\label{eq:help_prop3}
&\big\|\widetilde{x}^{(k+1)}-\|W\|^2\widetilde{x}^{(k-1)}\big\| =  \big\|\widetilde{x}^{(k+1)}-\widetilde{W}^2\widetilde{x}^{(k-1)}\big\|\nonumber \\
&\hspace{0.20in}\leq \|\widetilde{x}^{(k+1)}\hspace{-0.03in}- \widetilde{W}\widetilde{y}^{(k+1)}\| \hspace{-0.03in}+\hspace{-0.03in}\|\widetilde{W}\widetilde{y}^{(k+1)}\hspace{-0.03in}- \widetilde{W}\widetilde{x}^{(k)}\| \hspace{-0.03in}\nonumber \\
&\hspace{1.8in}+\|\widetilde{W}\widetilde{x}^{(k)}\hspace{-0.03in}- \widetilde{W}^2\widetilde{y}^{(k)}\|\hspace{-0.03in}+\hspace{-0.03in}\|\widetilde{W}^2\widetilde{y}^{(k)}- \widetilde{W}^2\widetilde{x}^{(k-1)}\|\nonumber\\
&\hspace{0.20in}\leq 2\|W\|\sqrt{\frac{1-c^2}{1-\beta^2}} \|{\widetilde{y}}^{(k+1)}\|+\frac{\sqrt{1-c^2}}{c}\|\widetilde{y}^{(k+1)}\|\nonumber \\
&\hspace{1.8in}+2\|W\|\sqrt{\frac{1-c^2}{1-\beta^2}} \|{\widetilde{y}}^{(k)}\|+\frac{\sqrt{1-c^2}}{c}\|\widetilde{y}^{(k)}\|.\nonumber \\
\end{align*}
Applying the observation 
$\|\widetilde{x}^{(k+1)}\|\leq \|\widetilde{y}^{(k+1)}\|\leq\|\widetilde{x}^{(k)}\|\leq \|\widetilde{y}^{(k)}\|\leq \|\widetilde{x}^{(k-1)}\|$
to this inequality, we obtain
\begin{equation}\label{eq:help_prop33}
\big\|\widetilde{x}^{(k+1)}-\|W\|^2\widetilde{x}^{(k-1)}\big\|
\leq 2\Big(2\|W\|\sqrt{\frac{1-c^2}{1-\beta^2}}+\frac{\sqrt{1-c^2}}{c}\Big)\|\widetilde{x}^{(k-1)}\|.
\end{equation}
The inverse triangle inequality applied to \eqref{eq:help_prop33} then yields 
\begin{align}\label{eq:help_2_prop33}
\|\widetilde{x}^{(k+1)}\|&\geq \|W\|^2\|\widetilde{x}^{(k-1)}\|-2\Big(2\|W\|\sqrt{\frac{1-c^2}{1-\beta^2}}+\frac{\sqrt{1-c^2}}{c}\Big)\|\widetilde{x}^{(k-1)}\| \nonumber \\
&= \left( \|W\|^2-2\Big(2\|W\|\sqrt{\frac{1-c^2}{1-\beta^2}}+\frac{\sqrt{1-c^2}}{c}\Big)\right)\|\widetilde{x}^{(k-1)}\|.
\end{align}
Assume $c$ is close enough to $1$ to ensure the RHS of \eqref{eq:help_2_prop33} is strictly positive. Note that, since the term in front of $\|\widetilde{x}^{(k-1)}\|$ converges to $\|W\|^2$ as $c$ approaches one, such $c$ exists.
Then, under this assumption on $c$, combining \eqref{eq:help_prop33} and \eqref{eq:help_2_prop33}, we have 
\begin{equation}\label{eq:contradiction4}
\frac{\|\widetilde{x}^{(k+1)}- \|W\|^2\widetilde{x}^{(k-1)}\|}{\|\widetilde{x}^{(k+1)}\|}\leq \frac{2\Big(2\|W\|\sqrt{\frac{1-c^2}{1-\beta^2}}+\frac{\sqrt{1-c^2}}{c}\Big)}{\|W\|^2-2\Big(2\|W\|\sqrt{\frac{1-c^2}{1-\beta^2}}+\frac{\sqrt{1-c^2}}{c}\Big)}.
\end{equation} 
We observe that the RHS of \eqref{eq:contradiction4} converges to zero as $c$ approaches one from the left.  
By the definition of $\widetilde{y}^{(k+2)}$ in \eqref{eq:tildey_reformulation}, its norm is bounded by the norm of any affine combination of $\widetilde{x}^{(k+1)}$ and $\widetilde{x}^{(k-1)}$, so \eqref{eq:contradiction4} implies 
\begin{align*}
\frac{\|\widetilde{y}^{(k+2)}\|}{\|\widetilde{x}^{(k+1)}\|}&\leq 
\frac{\Big\|\frac{1}{1-\|W\|^2}\widetilde{x}^{(k+1)}+\frac{-\|W\|^2}{1-\|W\|^2}\widetilde{x}^{(k-1)}\Big\|}{\|\widetilde{x}^{(k+1)}\|} \nonumber \\
&\leq \frac{1}{{1-\|W\|^2}}\left(\frac{2\Big(2\|W\|\sqrt{\frac{1-c^2}{1-\beta^2}}+\frac{\sqrt{1-c^2}}{c}\Big)}{{\|W\|^2}-2\Big(2\|W\|\sqrt{\frac{1-c^2}{1-\beta^2}}+\frac{\sqrt{1-c^2}}{c}\Big)}\right). \nonumber 
\end{align*}
As $c$ approaches one from the left, the RHS of the above equation goes to zero, but this contradicts \eqref{eq:contradiction2}, which states  $\|\widetilde{y}^{(k+2)}\|/\|\widetilde{x}^{(k+1)}\|\geq c$. 
Thus, \eqref{eq:contradiction1} holds, and the first part of the argument is completed.

The second part of the proposition for the case $m=1$ is proved by an example. Let $|w|<1$, $W:=\text{Diag}(w,-w)$, $a:={0}$, $\widetilde{x}^{(0)}=[u,v]$ such that $(1-w)u^2=(1+w)v^2$. Then AA(1) applied to this problem generates the sequences $\widetilde{y}^{(k)}=\widetilde{x}^{(k-1)}$ and $\widetilde{x}^{(k)}=W\widetilde{y}^{(k)}$ for all $k\geq 2$. Therefore, AA(1) is reduced to the fixed-point method, and it converges $r$-linearly with convergence factor $\|W\|$. 

\subsubsection{Proof of \eqref{eq:neary}}\label{sec:proof_neary}
There are two cases to consider. 

\textbf{Case 1: All eigenvalues of $W$ are $\pm\|W\|$.}
In this case, we see that  \eqref{eq:neary} holds with $\widetilde{W}=W$ because the LHS of this equation is zero. 

\textbf{Case 2: Not all eigenvalues of $W$ are $\pm\|W\|$.}
Let $L_1$ and $L_2$ be the eigenspaces of $W$ with eigenvalues $\|W\|$ and $-\|W\|$, respectively. Let $P_{L}(x)$ denote the projection of the vector $x$ onto the subspace $L$, with projection matrix $P_L$ such that $P_L x = P_L(x)$. Let  $L^\perp$ represent the complementary space of $L$, and   $L_1\oplus L_2$ be the direct sum of $L_1$ and $L_2$. Let $\{\lambda_i(W)\}_{i=1}^n$ and $\{v_i(W)\}_{i=1}^n$ be the eigenvalues and eigenvectors of $W$.
Then, 
\begin{align*}
\|W \widetilde{y}^{(k+i)}\|^2=&\sum_{1\leq j\leq n} \lambda_j^2(W)|v_j(W)^T\widetilde{y}^{(k+i)}|^2\\\leq& \|W\|^2\|P_{L_1\oplus L_2}\widetilde{y}^{(k+i)}\|^2+\max_{j: \lambda_j(W)\neq \pm\|W\|}\lambda_j^2(W)\|P_{(L_1\oplus L_2)^\perp}\widetilde{y}^{(k+i)}\|^2.
\end{align*}
Combining this inequality with the second inequality in \eqref{eq:contradiction2} and the definition of $\beta$, we obtain
\[
\|P_{L_1\oplus L_2}\widetilde{y}^{(k+i)}\|^2+\beta^2\|P_{(L_1\oplus L_2)^\perp}\widetilde{y}^{(k+i)}\|^2> c^2\|\widetilde{y}^{(k+i)}\|^2. 
\]
Utilizing the Pythagorean theorem we obtain
\[\|\widetilde{y}^{(k+i)}\|^2-(1-\beta^2)\|P_{(L_1\oplus L_2)^\perp}\widetilde{y}^{(k+i)}\|^2> c^2\|\widetilde{y}^{(k+i)}\|^2.
\]
A quick rearrangement of terms then yields  
\begin{equation}\label{eq:nearL}
\|P_{(L_1\oplus L_2)^\perp}\widetilde{y}^{(k+i)}\|^2< \frac{1-c^2}{1-\beta^2}\|\widetilde{y}^{(k+i)}\|^2.
\end{equation}
Let $\widetilde{W}$ share the same eigenvectors with $W$ but modify its eigenvalues by replacing any value not equal to $-\|W\|$ with $\|W\|$. 
Thus, it follows $\|\widetilde{W}\|=\|W\|$, and $\widetilde{W}^2=\|W\|^2I$. 
Furthermore, noting 
\begin{align}
P_{L_1}=\sum_{1\leq i\leq n: \lambda_i(W)=\|W\|}v_i(W)v_i(W)^T,
\end{align}
and $P_{L_2}$ and $P_{L_3}$ can be represented similarly for indices $\{1\leq i\leq n: \lambda_i(W)=-\|W\|\}$ and $\{1\leq i\leq n: \lambda_i(W)\neq\pm\|W\|\}$, we see that 
\begin{align*}
\widetilde{W}=&\|W\|(P_{L_1}-P_{L_2}+P_{L_3})\\=&
\|W\|\Big(\sum_{i=1, \lambda_i(W)\neq -\|W\|}^nv_i(W)v_i(W)^T-\sum_{i=1, \lambda_i(W)=-\|W\|}^nv_i(W)v_i(W)^T\Big),\end{align*}
where $L_3=(L_1\oplus L_2)^\perp.$ 
In addition,
\begin{align}\label{eq:nearL2}
W-\widetilde{W}&=(\|W\|P_{L_1}-\|W\|P_{L_2}+P_{L_3}WP_{L_3})-\|W\|(P_{L_1}-P_{L_2}+P_{L_3}) \\
&=P_{L_3}WP_{L_3}-\|W\|P_{L_3}, \nonumber 
\end{align}
and so
\begin{align*}
\| (W-\widetilde{W})\widetilde{y}^{(k+i)}\| 
=&\|(P_{L_3}WP_{L_3}-\|W\|P_{L_3})\widetilde{y}^{(k+i)}\| \nonumber \\
=&\|(P_{L_3}W\!-\!\|W\|I)P_{L_3}\widetilde{y}^{(k+i)}\|\\\leq& 2\|W\|\|P_{L_3}\widetilde{y}^{(k+i)}\| \nonumber 
\leq2\|W\|\|\widetilde{y}^{(k+i)}\|\sqrt{\frac{1-c^2}{1-\beta^2}},
\end{align*}
where the first equality applies \eqref{eq:nearL2}, the first inequality follows from the orthogonality of $P_{L_3}$, i.e., $\|P_{L_3}W\|\leq \|P_{L_3}\|\|W\|=\|W\|$, and the second inequality is a result of \eqref{eq:nearL}. 


{
\subsection{Proof of Theorem~\ref{thm:AA_invariance}} We prove the result by induction. 
Let $\bar{x}\in \R^n$, then from the definitions of $q_1$ and $q_2$ we see that:
\begin{equation}\label{eq:help_in_1}
q_2(\bar{x}+b)  
= q_1(\bar{x}) + b, 
\end{equation}
and
\begin{equation}\label{eq:help_in_2}
q_2(\bar{x} + b) - (\bar{x} +b) = q_1(\bar{x}) - \bar{x}. 
\end{equation}
Let $x^{(0)}\in \R^n$, and let $\{x_1^{(k)}\}$ and $\{x_2^{(k)}\}$ be the sequences generated by AA(m) initialized at $x_1^{(0)} = x^{(0)}$ and $x_2^{(0)} = x^{(0)}+b$ respectively. By \eqref{eq:help_in_1}, we observe  that $x_2^{(1)} = x_1^{(1)} + b$.
From \eqref{eq:aa_optimization} and \eqref{eq:help_in_2},  when $k=1$, the optimal $\alpha$ values for \eqref{eq:aa_optimization} are  identical for the two sequences. 
Thus, letting $\alpha_0$ and $\alpha_1$ be the optimal solution to \eqref{eq:aa_optimization} (for both sequences), we observe that 
\begin{align}
x_2^{(2)} &= \alpha_0\; q_2(x_2^{(0)}) + \alpha_1\; q_2(x_2^{(1)}) \nonumber \\
&= \alpha_0\; q_2(x_1^{(0)} +b) + \alpha_1\;  q_2(x_1^{(1)}+b) \nonumber \\
&= \alpha_0\left( q_1(x^{(0)}) + b\right) + \alpha_1 \left( q_1(x_1^{(1)}) + b \right) \nonumber \\
&=\alpha_0\; q_1(x^{(0)}) + \alpha_1\; q_1(x_1^{(1)}) + (\alpha_0 + \alpha_1)b  \nonumber \\
&= x_1^{(2)} + b, \nonumber 
\end{align}
where first equality follows from Step 2 of Algorithm~\ref{alg:aa}, the third equality uses \eqref{eq:help_in_1}, and the last uses the fact $\alpha_0 + \alpha_1 = 1$. 
So, $x_2^{(k)} = x_1^{(k)} + b$ holds for $k=0,1,2$, which establishes the base step for our inductive argument.

Now, assume $x_2^{(k)} = x_1^{(k)} + b$ for $k=1,\cdots,K$ with $K\geq 2$. Then \eqref{eq:help_in_2} ensures that the optimal solutions to \eqref{eq:aa_optimization}, i.e., $\{\alpha_j^{(k)}\}_{k=K-m_K}^K$, are identical when $k=K$ for the two sequences. 
Following our prior argument, we see 
\begin{align}
x_2^{(K+1)}&= \sum_{j=K-m_k}^{K} \alpha_j^{(K)}q_2(x_2^{(j)}) \nonumber \\ 
&= \sum_{j=K-m_k}^{K} \alpha_j^{(K)}q_2(x_1^{(j)}+b) \nonumber \\
&= \sum_{j=K-m_k}^{K} \alpha_j^{(K)}(q_1(x_1^{(j)}) + b) \nonumber \\
&= \sum_{j=K-m_k}^{K} \alpha_j^{(K)}q_1(x_1^{(j)}) + \left( \sum_{j=K-m_k}^{K} \alpha_j^{(K)}\right)b \nonumber \\
&=x_1^{(K+1)} + b,\nonumber 
\end{align}
and $x_2^{(k)} = x_1^{(k)} + b$ holds for $k=K+1$. Then, by induction,  Theorem~\ref{thm:AA_invariance} is proved.
}

\subsection{Proof of Theorem~\ref{thm:main2}}
We divide the proof into three sections. 
First, we present a reformulation of Anderson acceleration, similar to the approach in Section~\ref{sec:reformulate_AA}, along with its key properties. 
Second, we establish the convergence of AA(m). 
Finally, we determine its convergence factor, thus completing the proof. 
\vspace{0.1in}

\textbf{Part 1: A reformulation of Anderson acceleration and its key  properties.} 
First, without loss of generality, assume that the fixed point $x_*$ is $0$. If $x_*\neq 0$, then $0$ is a fixed point of $\widetilde{q}(x)=f(x+x_*)+x$, where $f(x)=q(x)-x$.  By applying Theorem \ref{thm:AA_invariance}, the sequences generated by AA(m), when applied to $\widetilde{q}(x)$ and $q(x)$, are equivalent up to a translation when initialized appropriately. Consequently, they exhibit the same convergence behavior. 
Thus, establishing the convergence factor of  $q$ is equivalent to establishing the convergence factor of AA(m) applied to $\widetilde{q}$. Therefore, it is sufficient to assume that the fixed point $x_*$ is $0$.

Second, instead of $x^{(k)}$, we consider $\widetilde{x}^{(k)}$, which is similar to the reformulation in Section \ref{sec:reformulate_AA} and is  defined as follows: 
\[
\widetilde{x}^{(k)}:=q(x^{(k)})-x^{(k)}.
\]
By Step 2 of Algorithm~\ref{alg:aa} and our modification at \eqref{eq:AA_minimize1}, the update formula of $\widetilde{x}^{(k)}$ can be written as: 
\begin{align}
\widetilde{x}^{(k+1)}&=(q-I)x^{(k+1)} \nonumber \\
&=(q-I)\big(\sum_{j=k-m_k}^{k}\alpha_j^{(k)} q(x^{(j)})\big) \nonumber \\
&=(q-I)\Big(\sum_{j=k-m_k}^{k}\alpha_j^{(k)}q\left((q-I)^{-1}(\widetilde{x}^{(j)})\right)\Big),\label{eq:update_nonlinear2}
\end{align}
where $\{\alpha_j^{(k)}\}_{j=k-m_k}^k$ is chosen such that it solves 
\[
\min\Big\|\sum_{j=k-m_k}^{k}\alpha_j^{(k)}\widetilde{x}^{(j)}\Big\|,\,\,\text{subject to }\sum_{j=k-m_k}^{k} \alpha_j^{(k)}=1, |\alpha_j^{(k)}|\leq C_0.
\]

Unlike the linear setting \eqref{eq:reformulation_linear_update}, where $q$ is a linear operator, this update formula  cannot be simplified to remove the expression $(q-I)^{-1}$. To facilitate the analysis, we define
\[\widetilde{y}^{(k+1)}:=\sum_{j=k-m_k}^{k}\alpha_j^{(k)}\widetilde{x}^{(j)}
\]
  and state the properties (P1) and (P2).  Property (P1)  establishes the relationship between  $\|\widetilde{y}^{(k+1)}\|$ and $ \|\widetilde{x}^{(k)}\|$, while  property (P2) extends the result $\widetilde{x}^{(k+1)}=q(\widetilde{x}^{(k+1)})=W\widetilde{y}^{(k+1)}$ in the linear setting by stating that their difference is of second-order magnitude. Their proofs are deferred to Section~\ref{sec:intermediate}. \vspace{0.05in}

\begin{enumerate}
\item[(P1)]  If $C_0\geq 1$, 
 then $\|\widetilde{y}^{(k+1)}\|\leq \|\widetilde{x}^{(k)}\|$ with  $\angle(O \widetilde{y}^{(k+1)} \widetilde{x}^{(k)} )\geq \pi/2$. 
\vspace{0.05in}

\item[(P2)]  There exists $c_1',M>0$ such that if $\max_{k-m_k\leq j\leq k}\|{\widetilde{x}}^{(j)}\|\leq c_1'$, then
\[
\|\widetilde{x}^{(k+1)}-W\widetilde{y}^{(k+1)}\| \leq   M\big(\max_{k-m_k\leq j\leq k}\|{\widetilde{x}}^{(j)}\|^2\big).
\]
\end{enumerate}
\vspace{0.1in}

\textbf{Part 2: Proof of local convergence.} We now prove AA(m) converges when the initialization is sufficiently close to $x_*=0$, in the sense that $\widetilde{x}^{(k)}$ converges to zero. 
Define the sequence $\{d^{(k)}\}_{k\geq 1}$ by 
\begin{equation}\text{$d^{(1)}=\|\widetilde{x}^{(1)}\|$ 
and $d^{(k+1)}=\|W\| d^{(k)}+ M(\max_{k-m_k\leq j\leq k}(d^{(k)})^2)$ for $k\geq 1$.}\label{eq:sequence_d}\end{equation} 
First, we claim that if 
\begin{equation}
d^{(1)}\leq \min\left(c_1, c_1', \frac{1-\|W\|}{2M}\|W\|^{m}\right),\label{eq:initial_d}
\end{equation} 
then, {for all $k\geq 1$}, we have 
\begin{equation}\label{eq:induction_d2}
\|W\|d^{(k)}\leq d^{(k+1)}\leq \frac{\|W\|+1}{2}d^{(k)}. 
\end{equation} 
The inductive argument proving  \eqref{eq:induction_d2} is as follows. {For the base case $k=1$, from \eqref{eq:sequence_d} we have $d^{(2)}=(\|W\|+Md^{(1)})d^{(1)}$, so \eqref{eq:induction_d2} holds if $d^{(1)}\leq (1-\|W\|)/2M$, as guaranteed by \eqref{eq:initial_d}.}

With the base case established, assume \eqref{eq:induction_d2} holds when $k=1,\cdots,K-1$. Then, we will prove it for $k=K$, which follows  because 
$\|W\|d^{(K)}\leq d^{(K+1)}$ and 
\begin{align*}
d^{(K+1)}=&\|W\| d^{(K)}+ M(\max_{K-m_K\leq j\leq K}(d^{(K)})^2) \\
=& \|W\| d^{(K)}+ M(d^{(K-m_K)})^2\\
\leq &\|W\| d^{(K)}+  M\Big(\frac{d^{(K)}}{\|W\|^{m_K}} d^{(1)}\Big) \\
=& \Big(\|W\|+\frac{Md^{(1)}}{\|W\|^{m_K}}\Big)d^{(K)}\\
\leq& \frac{\|W\|+1}{2}d^{(K)},
\end{align*}
where the second equality is a result of the inductive hypothesis, 
the first inequality follows by \eqref{eq:induction_d2} for $1\leq k\leq K-1$ to see $d^{(K)}\geq \|W\|^{m_K}d^{(K-m_K)}$ and $d^{(K-m_K)}\leq d^{(1)}$, and the last inequality applies \eqref{eq:initial_d}, $m\geq m_K$, and $\|W\|\leq 1$. This completes the proof of \eqref{eq:induction_d2}.

{


Next, we use another proof by induction to show that, when  \eqref{eq:initial_d} holds, \begin{equation}\label{eq:induction_d0}\text{$\|\widetilde{x}^{(k)}\| \leq d^{(k)}$ for all $k\geq 1$}.\end{equation} 
The inductive argument for \eqref{eq:induction_d0} is as follows. First, \eqref{eq:induction_d0} holds for the base case $k=1$, {by the definition of $d^{(1)}$}. Assume \eqref{eq:induction_d0} holds for all $k$ up to some $k_0 \geq 1$. 
We now prove \eqref{eq:induction_d0} holds for $k=k_0+1$: applying 
 (P1) and (P2) we have that when $\max_{k-m_k\leq j\leq k}\|{\widetilde{x}}^{(j)}\|\leq c_1'$, 
\begin{align}\label{eq:help_pf_eq}
\| \widetilde{x}^{(k+1)}\| \leq& \|W\widetilde{y}^{(k+1)}\|+  M\big(\max_{k-m_k\leq j\leq k}\|{\widetilde{x}}^{(j)}\|^2\big) \nonumber \\
\leq& \|W\| \|\widetilde{y}^{(k+1)}\|+  M\big(\max_{k-m_k\leq j\leq k}\|{\widetilde{x}}^{(j)}\|^2\big) \nonumber \\
\leq& \|W\| \|\widetilde{x}^{(k)}\|+  M\big(\max_{k-m_k\leq j\leq k}\|{\widetilde{x}}^{(j)}\|^2\big).
\end{align}
Observe that the assumption {$\max_{{k_0}-m_{k_0}\leq j\leq {k_0}}\|{\widetilde{x}}^{(j)}\|\leq c_1'$} holds because of \eqref{eq:initial_d} and the fact that $d^{(k)}$ is nonincreasing, which follows from \eqref{eq:induction_d2}.
Thus, 
\begin{align*}
\|\widetilde{x}^{(k_0+1)}\|\leq &\|W\| \|\widetilde{x}^{(k_0)}\|+  M\big(\max_{k_0-m_{k_0}\leq j\leq k_0}\|{\widetilde{x}}^{(j)}\|^2\big)
\\\leq& \|W\|d^{(k_0)}+M(\max_{k_0-m_{k_0}\leq j\leq k_0}(d^{(k_0)})^2)=d^{(k_0+1)}.
\end{align*}
}
This completes the proof of \eqref{eq:induction_d0}. 

Combining \eqref{eq:induction_d2} with \eqref{eq:induction_d0}, the convergence of $\|\widetilde{x}^{(k)}\|$ to zero is proved.
\vspace{0.1in}

\textbf{Part 3: Proof of the local convergence factor.} {In the third and final part of the proof, we show that the $r$-linear convergence factor of $\widetilde{x}^{(k)}$ is bounded above by $\sqrt{w_0\|W\|}$.} 
To accomplish this, it is sufficient to demonstrate that for any small $\epsilon>0$, there exists a sufficiently large $K_\epsilon$ such that for any $k\geq K_\epsilon$ there exists $k-(m+1)\leq k_1\leq k$ such that
\begin{equation}\label{eq:ratio3}
\frac{\|\widetilde{x}^{(k+1)}\|}{\|\widetilde{x}^{(k_1)}\|}\leq \Big((1+\epsilon)w_0\|W\|\Big)^{(k+1-k_1)/2}.
\end{equation}
This implies that the r-linear {convergence factor} is bounded above by $\sqrt{w_0\|W\|(1+\epsilon)}$: Using a similar argument as in \eqref{eq:r_linear_factor}, \eqref{eq:ratio3} implies that, for any $k>K_\epsilon$, there exists a sequence $k=k_0>k_1>\cdots>k_l> K_\epsilon\geq k_{l+1}$ such that for all $0\leq i\leq l$,  $1\leq k_i-k_{i+1}\leq m+2$ and 
\[
\frac{\|\widetilde{x}^{(k_i)}\|}{\|\widetilde{x}^{(k_{i+1})}\|}\leq \Big((1+\epsilon)w_0\|W\|\Big)^{(k_i-k_{i+1})/2}.
\]
As a result, when $\epsilon$ is chosen such that $(1+\epsilon)w_0\|W\|<1$, which is possible because $\|W\|<1$ and $w_0\leq \|W\|$, we have that
\begin{align*}
\|\widetilde{x}^{(k)}\| \leq& \Big((1+\epsilon)w_0 \|W\| \Big)^{(k-k_{l+1})/2}  
\|\widetilde{x}^{(k_{l+1})}\|\\
\leq& 
\Big( (1+\epsilon)w_0 \|W\| \Big)^{(k-K_\epsilon)/2 } \max\left(\|\widetilde{x}^{(K_\epsilon-m-1)}\|,\cdots,\|\widetilde{x}^{(K_\epsilon-1)}\|,\|\widetilde{x}^{(K_\epsilon)}\|  \right),
\end{align*}
where the last inequality applies $(1+\epsilon)w_0\|W\|<1$ and $k_{l+1}\geq K_{\epsilon}-m-1$.
Thus, the convergence factor of the sequence $\widetilde{x}^{(k)}$ is upper bounded by $\sqrt{w_0\|W\|(1+\epsilon)}$, and 
as we may choose $\epsilon$ as small as possible, the theorem will be proved provided we establish \eqref{eq:ratio3}. 

To prove \eqref{eq:ratio3}, we first assume 
\begin{equation}\label{eq:ratio3_assumption}
\max_{k-m-1\leq j\leq k}\|\widetilde{x}^{(j)}\|\leq (w_0\|W\|)^{-(m+2)/2}\|\widetilde{x}^{(k+1)}\|,
\end{equation}
since otherwise we would have $\max_{k-m-1\leq j\leq k}\|\widetilde{x}^{(j)}\|> (w_0\|W\|)^{-(m+2)/2}\|\widetilde{x}^{(k+1)}\|$ and  \eqref{eq:ratio3} holds immediately with $k_1=\arg\max_{k-m-1\leq j\leq k}\|\widetilde{x}^{(j)}\|$; i.e.,
\[
\frac{\|\widetilde{x}^{(k+1)}\|}{\|\widetilde{x}^{(k_1)}\|}=\frac{\|\widetilde{x}^{(k+1)}\|}{\max_{k-m-1\leq j\leq k}\|\widetilde{x}^{(j)}\|}< (w_0\|W\|)^{(m+2)/2}\leq(w_0\|W\|)^{(k+1-k_1)/2}.
\]
Assuming \eqref{eq:ratio3_assumption} holds and utilizing (P2), it follows that 
\begin{align}\label{eq:tildexk+1}
\|\widetilde{x}^{(k+1)}\| &\leq \|W\|\|\widetilde{y}^{(k+1)}\|+M(\max_{k-m_k\leq j\leq k}\|\widetilde{x}^{(j)}\|^2) \nonumber \\
&\leq \|W\|\|\widetilde{y}^{(k+1)}\|+M(w_0\|W\|)^{-(m+2)}\|\widetilde{x}^{(k+1)}\|^2. 
\end{align}
Letting $\epsilon_{k+1}:=M(w_0\|W\|)^{-(m+2)}\|\widetilde{x}^{(k+1)}\|$, we note  \eqref{eq:tildexk+1} is equivalent to 
\begin{align}\label{eq:tildexk+12}
\|\widetilde{x}^{(k+1)}\|\leq \frac{\|W\|}{1-\epsilon_{k+1}}\|\widetilde{y}^{(k+1)}\|.
\end{align}
Next, let $\widehat{y}^{(k+1)}$ denote the projection of the origin onto the line connecting $\widetilde{x}^{(k-1)}$ and $\widetilde{x}^{(k)}$; a visualization is shown in Figure~\ref{fig:visualization_proof4}. 
Now, we introduce two intermediates results: there exists $K_\epsilon>0$ such that for all $k>K_\epsilon$ and $C_0>C_0'$, 
\begin{equation}\label{eq:adaptproof1}
\|\widetilde{y}^{(k+1)}\|\leq \|\widehat{y}^{(k+1)}\|,\end{equation}  
and
\begin{equation}\label{eq:adaptproof2}
\|\widehat{y}^{(k+1)}\| \leq  w_0\Big(1+\frac{2{\|W\|\epsilon_{k+1}}}{(\lambda_{\max}(W)-\lambda_{\min}(W))({1-\epsilon_{k+1}})}\Big)\|\widetilde{x}^{(k-1)}\|.
\end{equation}
The proof of \eqref{eq:adaptproof1} and \eqref{eq:adaptproof2} are deferred to Section~\ref{sec:intermediate}. We note that they are adaptations of \eqref{eq:twoiterations2} and \eqref{eq:twoiterations3} in the proof of Theorem~\ref{thm:main1}.

Finally, by combining \eqref{eq:tildexk+12}, \eqref{eq:adaptproof1} and  \eqref{eq:adaptproof2},
we have
\begin{equation}\label{eq:ratio2}
\frac{\|\widetilde{x}^{(k+1)}\|}{\|\widetilde{x}^{(k-1)}\|}\leq \frac{w_0\|W\|\Big(1+\frac{2{\|W\|\epsilon_{k+1}}}{(\lambda_{\max}(W)-\lambda_{\min}(W))({1-\epsilon_{k+1}})}\Big)}{1-\epsilon_{k+1}}.
\end{equation}
Previously in Part 2, we proved that  $\|\widetilde{x}^{(k)}\|$ converges to zero. Thus, for sufficiently large $k$, $\epsilon_{k+1}\rightarrow 0$ and \eqref{eq:ratio3} holds with $k_1=k-1$, completing the argument. 

\subsubsection{Proof of auxiliary  results}\label{sec:intermediate} 
In this section, we prove all the intermediate results necessary for the proof of Theorem~\ref{thm:main2}.

\vspace{0.1in}
{\bf Proof of Property (P1).} The definition of $\widetilde{y}^{(k+1)}$ in the modified AA(m) can be written as 
\begin{equation}\label{eq:nonlinear_reformulation}
\widetilde{y}^{(k+1)}=\argmin_{ y\in\mathcal{S}_k}\|y\|,\,\,\text{where}\,\,\mathcal{S}_k:=\Big\{\sum_{j=k-m_k}^{k}\alpha_j\widetilde{x}^{(j)} \;\bigg|\; \sum_{j=k-m_k}^{k} \alpha_j=1,\; \|\alpha\|_{\infty}<C_0, \Big\}.
\end{equation}
Since $C_0\geq 1$, we have $\widetilde{x}^{(k)}\in \mathcal{S}_k$, as it corresponds to $\alpha_k=1$ and $\alpha_j=0$ for all $j<k$. It follows that $\|\widetilde{y}^{(k+1)}\|\leq \|\widetilde{x}^{(k)}\|$, thereby establishing the first part of (P1). 

We prove the second part of $\angle(O \widetilde{y}^{(k+1)} \widetilde{x}^{(k)} )\geq \pi/2$ by contradiction. Otherwise, we have 
 $\angle(O \widetilde{y}^{(k+1)} \widetilde{x}^{(k)} )<\pi/2$, then $\langle\widetilde{y}^{(k+1)},\widetilde{y}^{(k+1)}-\widetilde{x}^{(k)}\rangle>0$. 
For small $\zeta>0$, $\widehat{y}=\zeta\widetilde{x}^{(k)}+(1-\zeta)\widetilde{y}^{(k+1)}$ would have a smaller norm than $\widetilde{y}^{(k+1)}$ because
\[
\|\widehat{y}\|^2=\|\widetilde{y}^{(k+1)}\|^2+2\zeta\langle\widetilde{y}^{(k+1)},\widetilde{x}^{(k)}-\widetilde{y}^{(k+1)}\rangle + \zeta^2\|\widetilde{x}^{(k)}-\widetilde{y}^{(k+1)}\|^2.
\]
Since $\mathcal{S}_k$ is a convex set that contains both $\widetilde{x}^{(k)}$ and $\widetilde{y}^{(k+1)}$, $\widehat{y}$ lies in $\mathcal{S}_k$. This contradicts the definition of $\widetilde{y}^{(k+1)}$ in  \eqref{eq:nonlinear_reformulation}; therefore, $\angle(O \widetilde{y}^{(k+1)} \widetilde{x}^{(k)} )\geq \pi/2$, proving the second aspect of (P1).


\vspace{0.1in}
\textbf{Proof of Property (P2).}  The second property (P2) is proved using \eqref{eq:update_nonlinear2} and the first-order Taylor approximation of $q$ at $x_*=0$.
The proof also utilizes big-$\mathcal{O}$ notation, where $f=\mathcal{O}(g)$ denotes that $f$ is of the order of $g$, that is, there exists $M>0$ such that $|f|\leq M|g|$. 
Furthermore, we note that our assumptions \eqref{eq:assumption_main2} and \eqref{eq:assumption_main3} imply that  $q, (q-I)$, and $(q-I)^{-1}$ are all locally Lipschitz continuous and approximated by their first-order Taylor expansions at $x_*=0$. 
Thus, 
\begin{align}
\widetilde{x}^{(k+1)} &= (q-I)\Big(\sum_{j=k-m_k}^{k}\alpha_j^{(k)}q\Big[(q-I)^{-1}(\widetilde{x}^{(j)})\Big]\Big) \nonumber \\ 
&= (q-I)\Big(\sum_{j=k-m_k}^{k}\alpha_j^{(k)}q\Big[(W-I)^{-1}\widetilde{x}^{(j)}+\mathcal{O}(\|\widetilde{x}^{(j)}\|^2)\Big]\Big) \nonumber \\ 
&= (q-I)\Big(\sum_{j=k-m_k}^{k}\alpha_j^{(k)}q\Big[(W-I)^{-1}\widetilde{x}^{(j)}\Big]+\mathcal{O}(\|\widetilde{x}^{(j)}\|^2)\Big) \nonumber \\
&= (q-I)\Big(\sum_{j=k-m_k}^{k}\alpha_j^{(k)}W\Big[(W-I)^{-1}\widetilde{x}^{(j)}\Big]+\mathcal{O}(\|\widetilde{x}^{(j)}\|^2)\Big) \nonumber \\
&= (q-I)\Big(\sum_{j=k-m_k}^{k}\alpha_j^{(k)}W\Big[(W-I)^{-1}\widetilde{x}^{(j)}\Big]\Big)+\mathcal{O}(\max_{k-m_k\leq j\leq k}\|\widetilde{x}^{(j)}\|^2)  \nonumber \\\nonumber
&= (W-I)\Big(\sum_{j=k-m_k}^{k}\hspace{0.0in}\alpha_j^{(k)}W\Big[(W-I)^{-1}\widetilde{x}^{(j)}\Big]\Big)+\mathcal{O}(\max_{k-m_k\leq j\leq k}\|\widetilde{x}^{(j)}\|^2)\\&\,\,\,\,\,\,+\mathcal{O}\Big(\Big\|\sum_{j=k-m_k}^{k}\alpha_j^{(k)}W\Big[(W-I)^{-1}\widetilde{x}^{(j)}\Big]\Big\|^2\Big)  \nonumber \\\nonumber
&=\hspace{-0.1in}\sum_{j=k-m_k}^{k}\hspace{-0.1in}\alpha_j^{(k)}W\widetilde{x}^{(j)}\hspace{-0.05in}+\mathcal{O}(\max_{k-m_k\leq j\leq k}\|\widetilde{x}^{(j)}\|^2)=W \widetilde{y}^{(k+1)} + \mathcal{O}\Big(\max_{k-m_k\leq j\leq k}\|\widetilde{x}^{(j)}\|\Big)^2, 
\end{align}
where the second, fourth and sixth equalities follow from local approximation results \eqref{eq:assumption_main2}, and the third and fifth equalities follow from Lipschitz continuity \eqref{eq:assumption_main3}. In addition, $c_1'$ is chosen to be small enough to ensure that the functions $q$, $(q-I)$  and $(q-I)^{-1}$ are approximated in the neighborhood $B(0,c_1)$ such that \eqref{eq:assumption_main2} and \eqref{eq:assumption_main3} hold. 
Note, the constraints $|\alpha_j^{(k)}|\leq C_0$ in \eqref{eq:AA_minimize1} are needed to guarantee the existence of such a $c_1'$.
Recognizing the constant $M>0$ follows by the definition of big-$\mathcal{O}$ completes the argument.

\vspace{0.1in}
\textbf{Proof of  \eqref{eq:adaptproof1}}. By the definition of $\widetilde{y}^{(k+1)}$ in \eqref{eq:nonlinear_reformulation}, it is sufficient to show that $\widehat{y}^{(k+1)}\in \mathcal{S}_k$. Since $\widehat{y}^{(k+1)}$, $\widetilde{x}^{(k)}$, and $\widetilde{x}^{(k-1)}$ are collinear, we have \[
\widehat{y}^{(k+1)}=\mu(\widetilde{x}^{(k)}-\widetilde{x}^{(k-1)})+\widetilde{x}^{(k)},\,\,\text{where $\mu=\frac{\|\widehat{y}^{(k+1)}-\widetilde{x}^{(k)}\|}{\|\widetilde{x}^{(k)}-\widetilde{x}^{(k-1)}\|}$}. 
\]
As a result, if 
\begin{equation}
\frac{\|\widehat{y}^{(k+1)}-\widetilde{x}^{(k)}\|}{\|\widetilde{x}^{(k)}-\widetilde{x}^{(k-1)}\|}<C_0-1,\label{eq:step1}
\end{equation} 
then $\widehat{y}^{(k+1)}\in \mathcal{S}_k$ and \eqref{eq:adaptproof1} is proved. 


{The remainder of the proof of \eqref{eq:adaptproof1} establishes \eqref{eq:step1} by estimating the numerator and denominator of the  LHS of \eqref{eq:step1}, ${\|\widehat{y}^{(k+1)}-\widetilde{x}^{(k)}\|}$ and ${\|\widetilde{x}^{(k)}-\widetilde{x}^{(k-1)}\|}$, separately. For this estimation, we claim 
\begin{align}\label{eq:tildexk+3}
\|\widetilde{x}^{(k)}\|\leq  \tau_{k+1}\|\widetilde{y}^{(k)}\|,\,\,\text{where}\,\,\,\tau_{k+1}=\frac{\|W\|}{1-\frac{\|W\|\epsilon_{k+1}}{1-\epsilon_{k+1}}}, 
\end{align}
and defer its proof to \eqref{eq:tildexk+5}-\eqref{eq:tildexk+6}.

By combining \eqref{eq:tildexk+3} with the observation $\|\widehat{y}^{(k+1)}\|\leq \|\widetilde{x}^{(k)}\|$, which holds since $\widehat{y}^{(k+1)}$ is the projection of the origin onto the line connecting $\widetilde{x}^{(k-1)}$ and $\widetilde{x}^{(k)}$, we obtain 
\begin{align}\label{eq:thm4_num2}
\|\widetilde{x}^{(k)}-\widehat{y}^{(k+1)}\| &\leq \|\widetilde{x}^{(k)}\|+\|\widehat{y}^{(k+1)}\|  \leq 2\|\widetilde{x}^{(k)}\|
\leq 2\tau_{k+1}\|\widetilde{y}^{(k)}\|.
\end{align}
From (P1) we have $\|\widetilde{x}^{(k-1)}\|\geq \|\widetilde{y}^{(k)}\|$, and leveraging \eqref{eq:tildexk+3} again, we see
\begin{align}\label{eq:thm4_den}
\|\widetilde{x}^{(k)}-\widetilde{x}^{(k-1)}\|\geq \|\widetilde{x}^{(k-1)}\|-\|\widetilde{x}^{(k)}\| 
\geq (1-\tau_{k+1})\|\widetilde{y}^{(k)}\|.
\end{align}
\begin{figure}[t]
\begin{center}
\includegraphics[scale=0.27]{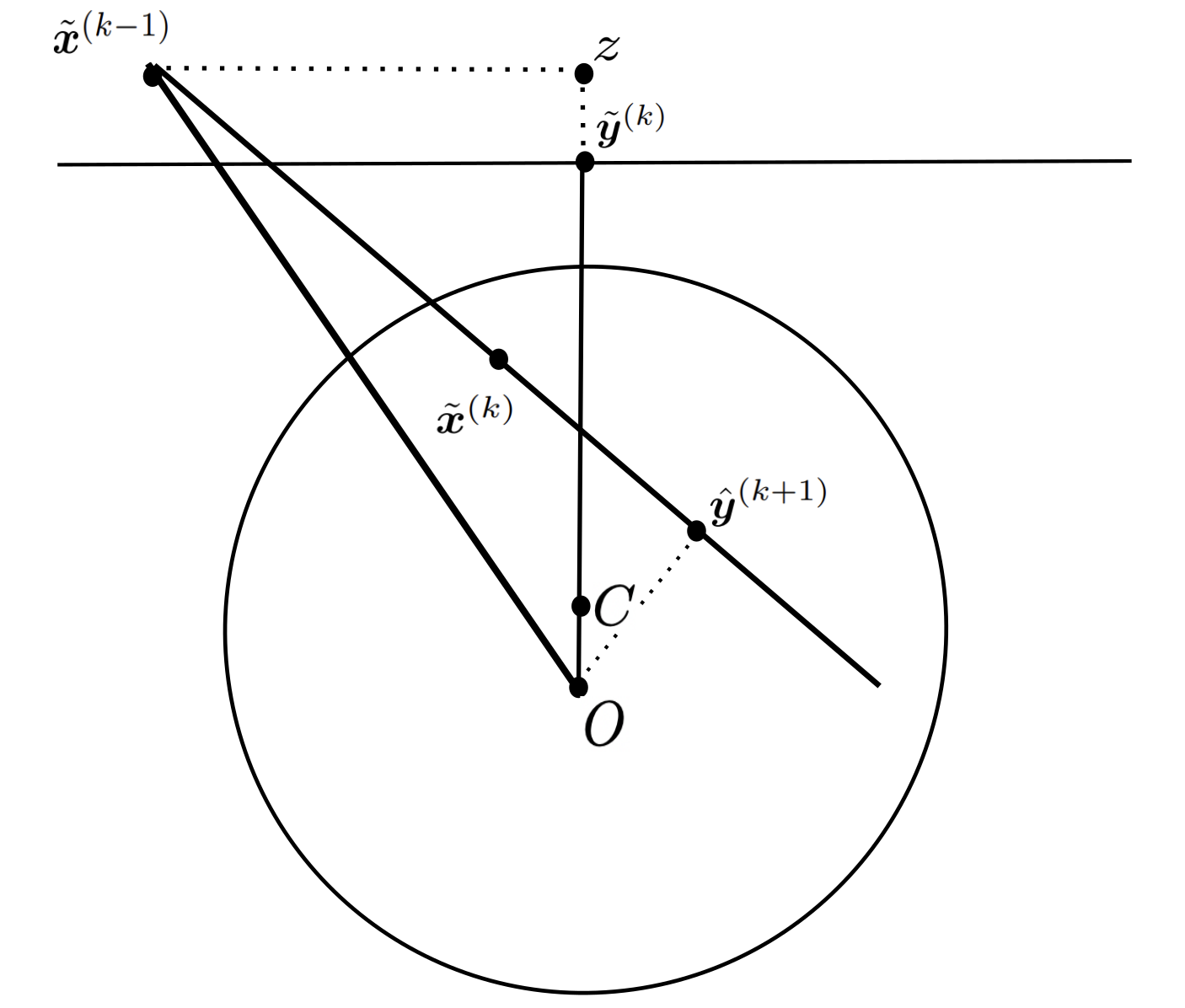}
\caption{A visualization used in Part 3 of the proof of Theorem~\ref{thm:main2}. Compared with Figure~\ref{fig:visualization_proof}, here $\angle(O \widetilde{y}^{(k+1)} \widetilde{x}^{(k)} ) \geq \pi/2$ instead of being equal.  }\label{fig:visualization_proof4}
\end{center}
\end{figure}
Combining \eqref{eq:thm4_num2} and \eqref{eq:thm4_den} to upper bound the LHS of \eqref{eq:step1} we obtain
\[
\frac{\|\widehat{y}^{(k+1)}-\widetilde{x}^{(k)}\|}{\|\widetilde{x}^{(k)}-\widetilde{x}^{(k-1)}\|} \leq \frac{2\tau_{k+1}}{1-\tau_{k+1}}.
\]
Note that we have proved the convergence of $\|\widetilde{x}^{(k)}\|$ to zero, we have $\epsilon_k\rightarrow 0$ and $\tau_k\rightarrow \|W\|$, as $k\rightarrow\infty$. 
Thus, when $C_0 > \frac{2\|W\|}{1-\|W\|}+1$, the RHS of the above inequality shall be bounded above by $C_0-1$, for sufficiently large $k$, and \eqref{eq:step1} shall hold.

It remains to prove \eqref{eq:tildexk+3}. Following an  argument similar to  \eqref{eq:tildexk+1} and \eqref{eq:tildexk+12} (with $k$ replacing $k+1$), (P2) and \eqref{eq:ratio3_assumption} imply
\begin{align}\label{eq:tildexk+5}
\|\widetilde{x}^{(k)}\| &\leq \|W\|\|\widetilde{y}^{(k)}\|+M(\max_{k-1-m_{k-1}\leq j\leq k-1}\|\widetilde{x}^{(j)}\|^2) \nonumber \\
&\leq \|W\|\|\widetilde{y}^{(k)}\|+M(w_0\|W\|)^{-(m+2)}\|\widetilde{x}^{(k+1)}\|^2=\|W\|\|\widetilde{y}^{(k)}\|+\epsilon_{k+1}\|\widetilde{x}^{(k+1)}\|. 
\end{align}
In particular, the first inequality in \eqref{eq:tildexk+5} follows from (P2) and the second inequality follows from \eqref{eq:ratio3_assumption}. 
Combining \eqref{eq:tildexk+5} with the upper bound  $\|\widetilde{x}^{(k+1)}\|\leq\frac{\|W\|}{1-\epsilon_{k+1}}\|\widetilde{y}^{(k+1)}\|$ derived from \eqref{eq:tildexk+12} and $\|\widetilde{y}^{(k+1)}\|\leq \|\widetilde{x}^{(k)}\|$ from (P1), we have
\begin{align}\label{eq:tildexk+6}
\|\widetilde{x}^{(k)}\| 
\leq&\|W\|\|\widetilde{y}^{(k)}\|+\frac{\|W\|\epsilon_{k+1}}{1-\epsilon_{k+1}}\|\widetilde{x}^{(k)}\|,
\end{align}
which proves \eqref{eq:tildexk+3}.
}


\vspace{0.1in}
\textbf{Proof of \eqref{eq:adaptproof2}}.
This proof follows a similar approach to the argument for Theorem \ref{thm:main1}. To extend  Lemma~\ref{lemma:bound} to our nonlinear setting, we apply the arguments of \eqref{eq:tildexk+5} and \eqref{eq:tildexk+6} as follows: {
\begin{align}\label{eq:tildexk+7}
\|\widetilde{x}^{(k)}-W\widetilde{y}^{(k)}\|\leq \epsilon_{k+1}\|\widetilde{x}^{(k+1)}\|\leq \frac{\|W\|\epsilon_{k+1}}{1-\epsilon_{k+1}}\|\widetilde{x}^{(k)}\|.
\end{align}
The equation above implies that  $\|\widetilde{x}^{(k)}\|-\|W\|\|\widetilde{y}^{(k)}\|\leq \frac{\|W\|\epsilon_{k+1}}{1-\epsilon_{k+1}}\|\widetilde{x}^{(k)}\|$. Thus, if $\|W\|+\frac{\|W\|\epsilon_{k+1}}{1-\epsilon_{k+1}}<1$ (a condition holds for large $k$ since $\lim_{k\rightarrow\infty}\epsilon_k=0$), we have  $\|\widetilde{x}^{(k)}\|\leq \|\widetilde{y}^{(k)}\|$ and 
\[
\|\widetilde{x}^{(k)}-W\widetilde{y}^{(k)}\|\\\leq \frac{\|W\|\epsilon_{k+1}}{1-\epsilon_{k+1}}\|\widetilde{y}^{(k)}\|.
\]
}
Following an argument similar to  \eqref{eq:radius_argument}, we have
\begin{align}
&\big\|\widetilde{x}^{(k)} - \frac{1}{2}(\lambda_{\max}(W) +\lambda_{\min}(W))\widetilde{y}^{(k)} \big\| \nonumber \\ 
&\hspace{0.5in}\leq \| \widetilde{x}^{(k)}-W\widetilde{y}^{(k)}\| + \| W\widetilde{y}^{(k)} - \frac{1}{2}(\lambda_{\max}(W) +\lambda_{\min}(W))\widetilde{y}^{(k)} \| \nonumber \\
&\hspace{0.5in}\leq \Big( \frac{\|W\|\epsilon_{k+1}}{1-\epsilon_{k+1}}+\frac{1}{2}(\lambda_{\max}(W) -\lambda_{\min}(W))\Big)\|\widetilde{y}^{(k)}\|. \label{eq:radius_argument2}
\end{align}
Thus, $\widetilde{x}^{(k)}$ lies within a ball centered at $\frac{1}{2}(\lambda_{\max}(W) +\lambda_{\min}(W))\widetilde{y}^{(k)}$, as before, but with the larger radius 
$r=\Big( \frac{\|W\|\epsilon_{k+1}}{1-\epsilon_{k+1}}+\frac{1}{2}(\lambda_{\max}(W) -\lambda_{\min}(W))\Big)\|\widetilde{y}^{(k)}\|$.  

A technical issue in adapting  the proof of Theorem \ref{thm:main1} to this setting is that (P1) implies $\angle(O \widetilde{y}^{(k)} \widetilde{x}^{(k-1)} )\geq \pi/2$, whereas the proof of Theorem \ref{thm:main1} relies on the more restrictive condition that $\angle(O \widetilde{y}^{(k)} \widetilde{x}^{(k-1)} )= \pi/2$. 
 We refer the readers to the visualizations in Figure~\ref{fig:visualization_proof}  and  Figure~\ref{fig:visualization_proof4} for a comparison. 
{Applying the same coordinate transformation as in the proof of Theorem~\ref{thm:main1}, we may set $\widetilde{y}^{(k)} = [0,1,0,\cdots,0]$ and $\widetilde{x}^{(k-1)}=[-1/\delta,b,0,\cdots,0]$,} where $b\geq 1$ because   $\angle(O \widetilde{y}^{(k)} \widetilde{x}^{(k-1)} )\geq \pi/2$ in (P1), rather than  $\angle(O \widetilde{y}^{(k)} \widetilde{x}^{(k-1)} )= \pi/2$ in the proof of Theorem \ref{thm:main1}.

To simplify notation, we rewrite the definition of \eqref{eq:w0} as follows:\begin{align}\label{eq:w01}
w_0 &= \sup_{\delta\geq 0}\;\;\sin\!\!\Big[\!
\sin^{-1}\!\!\Big(\phi(\delta)\Big)\! 
+ \theta(\delta)\Big] \\
&=\sup_{\delta\geq 0}\Big[\phi(\delta)\cos\Big(\theta(\delta)\Big)+\sqrt{1-\phi^2(\delta)}\sin\Big(\theta(\delta)\Big)\Big], \nonumber 
\end{align}
where $\phi(\delta)$ and $\theta(\delta)$ are given by 
\begin{align*}
\phi(\delta)&:=\!\frac{|\lambda_{\max}(W)-\lambda_{\min}(W)|\delta}{\sqrt{\!4+\!\big(2\!-\!\lambda_{\max}(\!W\!)\!-\!\lambda_{\min}(W\!)\big)^2\!\delta^2}}\!,\\\theta(\delta)&:=\Big|\tan^{-1}(\delta)-\tan^{-1}\Big(\Big(1-\frac{\lambda_{\max}(W)+\lambda_{\min}(W)}{2}\Big)\delta\Big)\Big|.
\end{align*}
Now, we have the adaption of \eqref{eq:geometric2}-\eqref{eq:geometric3} as follow:  for \begin{equation}
t=1+\frac{2{\|W\|\epsilon_{k+1}}}{(\lambda_{\max}(W)-\lambda_{\min}(W))({1-\epsilon_{k+1}})},\label{eq:t_definition}
\end{equation}
\begin{align}
&\nonumber\sin\angle(C\widetilde{x}^{(k-1)}\widetilde{x}^{(k)})\leq \frac{\Big(2\frac{\|W\|\epsilon_{k+1}}{1-\epsilon_{k+1}}+|\lambda_{\max}(W)-\lambda_{\min}(W)|\Big)\delta}{\sqrt{4+\big(2b-\lambda_{\max}(W)-\lambda_{\min}(W)\big)^2\delta^2}}\\
&\hspace{0.5in}\leq \frac{\Big(2\frac{\|W\|\epsilon_{k+1}}{1-\epsilon_{k+1}}+|\lambda_{\max}(W)-\lambda_{\min}(W)|\Big)\delta}{\sqrt{4+\big(2-\lambda_{\max}(W)-\lambda_{\min}(W)\big)^2\delta^2}}
= t\phi(\delta)\label{eq:geometric5},
\end{align}
where the second inequality follows from $b\geq 1$.

Set $z = [0,b,0,\cdots,0]$, we have the following adaption of  \eqref{eq:geometric4}: 
\begin{align}\nonumber
&\angle(O\widetilde{x}^{(k-1)}C)=|\angle(O\widetilde{x}^{(k-1)}z)-\angle(C\widetilde{x}^{(k-1)}z)|\\
&\hspace{0.2in}=\bigg|\tan^{-1}(b\delta)-\tan^{-1}\Big(\Big(b-\frac{\lambda_{\max}(W)+\lambda_{\min}(W)}{2}\Big)\delta\Big)\bigg|\leq \theta(\delta),
\label{eq:geometric6}
\end{align}
where last inequality of \eqref{eq:geometric6} is argued as follows: since $\frac{d}{d x}\tan^{-1}x=\frac{1}{1+x^2}$
is nonincreasing, for any $\delta_1>\delta_2>0$ and $b>0$, $\tan^{-1}(\delta_1)-\tan^{-1}(\delta_2)\geq \tan^{-1}(b+\delta_1)-\tan^{-1}(b+\delta_2)$. So,  $\bigg|\tan^{-1}(b\delta)-\tan^{-1}\Big(\Big(b-\frac{\lambda_{\max}(W)+\lambda_{\min}(W)}{2}\Big)\delta\Big)\bigg|$ is upper bounded by its value when $b=1$, which equals $\theta(\delta)$. 

Finally, by \eqref{eq:geometric5}, \eqref{eq:geometric6}, the triangle inequality, and $t\geq 1$, we obtain \begin{align*}
\sin\angle(O\widetilde{x}^{(k-1)}\widetilde{x}^{(k)})\leq &\sin(\angle(C\widetilde{x}^{(k-1)}\widetilde{x}^{(k)})+\angle(O\widetilde{x}^{(k-1)}C))\\\leq &  \sin\Big(\sin^{-1}\big(t\phi(\delta)\big)+\theta(\delta)\Big) \\=& t  \phi(\delta) \cos\big(\theta(\delta)\big)\!+\!\sqrt{1\!-\!t^2\phi^2(\delta)}\sin\big(\theta(\delta)\big)\\\leq & t  \phi(\delta) \cos\big(\theta(\delta)\big)+t\sqrt{1\!-\!\phi^2(\delta)}\sin\big(\theta(\delta)\big)\leq t w_0,
\end{align*}
and \eqref{eq:adaptproof2} is established using  $\|\widehat{y}^{(k+1)}\|/\|\widetilde{x}^{(k-1)}\|=\sin\angle(O\widetilde{x}^{(k-1)}\widetilde{x}^{(k)})$ and the definition of $t$ in \eqref{eq:t_definition}.
} 

\section{Conclusion}\label{sec:conclusion}
This paper studies the convergence properties of the classical windowed AA algorithm and provides the first argument showing it improves the {\color{black}r}-linear convergence factor of the fixed-point iteration when $q$ is linear and symmetric. 
We extend our investigation to nonlinear operators $q$ possessing a symmetric Jacobian at fixed points and prove an adapted version of the method locally enjoys improved {\color{black}r}-linear convergence. 
Our theoretical discoveries were validated through simulations where we showcased AA(m) significantly outperforms the fixed-point iteration for TME. 
In the future, we hope to generalize our results to non-symmetric linear operators and refine the estimations of the $r$-linear convergence factor.

\subsubsection*{Acknowledgements}
First, GL thanks Yousef Saad for introducing him to Anderson acceleration. The authors also thank and acknowledge their funding sources as this material is based upon work supported by the National Science Foundation Graduate Research
Fellowship Program under Grant No.~2237827, NSF Award DMS-2152766, and NSF DMS-2318926. 
Any opinions, findings, and conclusions or recommendations expressed in this material are those of the author(s) and do not necessarily reflect the views of the National Science Foundation. 


\end{document}